\documentclass[a4paper]{amsart}

\usepackage[T1]{fontenc}
\usepackage{amsmath, amssymb, amsthm, mathrsfs, mathtools, graphicx}
\usepackage{varioref}
\usepackage[backref=page]{hyperref}
\usepackage{enumitem, xspace, ifthen, comment}
\usepackage[all]{xy}
\usepackage[nameinlink, capitalize]{cleveref}
\usepackage[dvipsnames]{xcolor}
\usepackage{tikz}

\setlist[enumerate]{
  label=(\thethm.\arabic*),
  before={\setcounter{enumi}{\value{equation}}},
  after={\setcounter{equation}{\value{enumi}}},
  itemsep=1ex
}

\setlist[itemize]{
  leftmargin=*,
  topsep=1ex,
  itemsep=1ex,
  label=$\circ$
}

\newtheorem*{thm-plain}{Theorem}
\newtheorem{thm}{Theorem}[section]
\newtheorem{lem}[thm]{Lemma}
\newtheorem{prp}[thm]{Proposition}
\newtheorem{cor}[thm]{Corollary}

\numberwithin{equation}{thm}

\newtheorem{bigthm}{Theorem}
\newtheorem{bigcor}[bigthm]{Corollary}

\theoremstyle{definition}
\newtheorem{dfn}[thm]{Definition}

\newtheorem*{dfn-plain}{Definition}

\theoremstyle{remark}
\newtheorem{clm}[thm]{Claim}

\newtheorem{awlog}[thm]{Additional Assumption}

\newtheorem{setup}[thm]{Setup}
\newtheorem{rem}[thm]{Remark}

\newtheorem{exm}[thm]{Example}
\newtheorem*{rem-plain}{Remark}

\newcommand{\inv}{^{-1}}
\newcommand{\from}{\colon}

\newcommand{\imp}{\Rightarrow}
\newcommand{\lto}{\longrightarrow}

\newcommand{\x}{\times}
\newcommand{\inj}{\hookrightarrow}

\newcommand{\bij}{\xrightarrow{\,\smash{\raisebox{-.5ex}{\ensuremath{\scriptstyle\sim}}}\,}}
\newcommand{\isom}{\cong}
\newcommand{\defn}{\coloneqq}

\newcommand{\tensor}{\otimes}

\newcommand{\wt}{\widetilde}
\newcommand{\wb}{\overline}
\newcommand{\wh}{\widehat}

\renewcommand{\d}{\mathrm d}
\newcommand{\del}{\partial}

\newcommand{\ddc}{dd^c}

\newcommand{\dual}{^{\smash{\scalebox{.7}[1.4]{\rotatebox{90}{\textup\guilsinglleft}}}}}
\newcommand{\ddual}{^{\smash{\scalebox{.7}[1.4]{\rotatebox{90}{\textup\guilsinglleft} \hspace{-.5em} \rotatebox{90}{\textup\guilsinglleft}}}}}
\newcommand{\acts}{\ \rotatebox[origin=c]{-90}{\ensuremath{\circlearrowleft}}\ }

\newcommand{\factor}[2]{\left. \raise 2pt\hbox{$#1$} \right/\hskip -2pt \raise -2pt\hbox{$#2$}}

\newdir{ ir}{{}*!/-5pt/@^{(}} 
\newdir{ il}{{}*!/-5pt/@_{(}} 

\DeclareMathOperator{\img}{im}

\DeclareMathOperator{\rk}{rk}

\DeclareMathOperator{\Aut}{Aut}

\newcommand{\set}[1]{\left\{ #1 \right\}}

\def\rd#1.{\lfloor{#1}\rfloor}
\def\rp#1.{\lceil{#1}\rceil}
\def\tw#1.{\langle{#1}\rangle}

\newcommand{\la}{\langle}
\newcommand{\ra}{\rangle}

\renewcommand{\O}[1]{\mathscr{O}_{#1}}
\newcommand{\Omegap}[2]{\Omega_{#1}^{#2}}
\newcommand{\Omegar}[2]{\Omega_{#1}^{[#2]}}
\newcommand{\T}[1]{\mathscr{T}_{#1}}
\newcommand{\can}[1]{\omega_{#1}}

\newcommand{\Reg}[1]{{#1}_{\mathrm{reg}}}
\newcommand{\Sing}[1]{{#1}_{\mathrm{sg}}}

\newcommand{\codim}[2]{\mathrm{codim}_{#1}(#2)}

\newcommand{\cc}[2]{\mathrm{c}_{#1}(#2)}

\def\Hnought#1.#2.{\mathit{\Gamma} \!\left( #1, #2 \right)}
\def\HH#1.#2.#3.{\mathrm{H}^{#1} \!\left( #2, #3 \right)}
\def\hh#1.#2.#3.{h^{#1} \!\left( #2, #3 \right)}
\def\RR#1.#2.#3.{R^{#1} #2_* #3}
\def\HHc#1.#2.#3.{\mathrm{H}_{\mathrm{c}}^{#1} \!\left( #2, #3 \right)}
\def\Hh#1.#2.#3.{\mathrm{H}_{#1} \!\left( #2, #3 \right)}
\def\Hom#1.#2.{\mathrm{Hom} \!\left( #1, #2 \right)}
\def\End#1.{\mathrm{End} \!\left( #1 \right)}
\def\sHom#1.#2.{\mathscr{H}\!om \!\left( #1, #2 \right)}
\def\Ext#1.#2.#3.{\mathrm{Ext}^{#1} \!\left( #2, #3 \right)}
\def\sExt#1.#2.#3.{\mathscr{E}\!xt^{#1} \!\left( #2, #3 \right)}
\def\Link#1.#2.{\mathrm{Link} \!\left( #1, #2 \right)}

\newcommand{\piet}[1]{\hat\pi_1(#1)}

\newcommand{\GL}[2]{\mathrm{GL}(#1, #2)}

\newcommand{\kahler}{K{\"{a}}hler\xspace}

\newcommand{\lt}{locally trivial\xspace}
\newcommand{\gkp}{maximally quasi-\'etale\xspace}
\newcommand{\qe}{quasi-\'etale\xspace}
\newcommand{\maxi}{\mathrm{max}}

\DeclareMathOperator{\Autn}{Aut^\circ}

\DeclareMathOperator{\Exc}{Exc}

\DeclareMathOperator{\Alb}{Alb}
\DeclareMathOperator{\alb}{alb}

\DeclareMathOperator{\Ric}{Ric}
\DeclareMathOperator{\tr}{tr}

\newcommand{\eps}{\varepsilon}
\renewcommand{\theta}{\vartheta}
\renewcommand{\phi}{\varphi}
\newcommand{\vp}{\varphi}
\newcommand{\om}{\omega}
\newcommand{\omte}{\omega_{t, \eps}}

\newcommand{\N}{\ensuremath{\mathbb N}}
\newcommand{\Z}{\ensuremath{\mathbb Z}}
\newcommand{\Q}{\ensuremath{\mathbb Q}}
\newcommand{\R}{\ensuremath{\mathbb R}}
\newcommand{\C}{\ensuremath{\mathbb C}}

\newcommand{\bS}{\ensuremath{\mathbb S}}

\newcommand{\sA}{\mathscr A}  
 \newcommand{\sE}{\mathscr E} \newcommand{\sF}{\mathscr F}
\newcommand{\sG}{\mathscr G}

\newcommand{\sS}{\mathscr S} \newcommand{\sT}{\mathscr T}

\definecolor{forrest}{RGB}{81,133,49}
\definecolor{mydarkblue}{RGB}{10,92,153}

\title[K{\"{A}}hler spaces with zero first Chern class]{K{\"{A}}hler spaces with zero first Chern class: Bochner principle, Albanese map and fundamental groups}

\author[Claudon]{Beno\^it Claudon}
\address{Univ Rennes, CNRS, IRMAR --- UMR 6625, F--35000 Rennes, France et Institut Universitaire de France}
\email{\href{mailto:benoit.claudon@univ-rennes1.fr}{benoit.claudon@univ-rennes1.fr}}
\urladdr{\href{https://perso.univ-rennes1.fr/benoit.claudon/}{perso.univ-rennes1.fr/benoit.claudon/}}

\author[Graf]{Patrick Graf}
\address{Lehrstuhl f\"ur Mathematik I, Universit\"at Bayreuth, 95440 Bayreuth, Germany}
\email{\href{mailto:patrick.graf@uni-bayreuth.de}{patrick.graf@uni-bayreuth.de}}
\urladdr{\href{http://www.pgraf.uni-bayreuth.de/en/}{www.graficland.uni-bayreuth.de}}

\author[Guenancia]{Henri Guenancia}
\address{Institut de Math\'ematiques de Toulouse, Universit\'e Paul Sabatier, 31062 Toulouse Cedex 9, France}
\email{\href{mailto:henri.guenancia@math.cnrs.fr}{henri.guenancia@math.cnrs.fr}}
\urladdr{\href{https://hguenancia.perso.math.cnrs.fr/}{hguenancia.perso.math.cnrs.fr/}}

\author[Naumann]{Philipp Naumann}
\address{Lehrstuhl f\"ur Mathematik VIII, Universit\"at Bayreuth, 95440 Bayreuth, Germany}
\email{\href{mailto:philipp.naumann@uni-bayreuth.de}{philipp.naumann@uni-bayreuth.de}}
\urladdr{\href{https://www.mathe8.uni-bayreuth.de/de/team/naumann/index.php}{www.pnaumann.uni-bayreuth.de}}

\date{\today}
\thanks{H.G.~was partially supported by the ANR project GRACK.
B.C.~was partially supported by the ANR projects Foliage ANR--16--CE40--0008 and Hodgefun ANR--16--CE40--0011.}
\keywords{\kahler spaces, klt singularities, vanishing first Chern class, holonomy, Albanese map, fundamental groups, Kodaira problem, decomposition theorem}
\subjclass[2010]{32J27, 14E30, 14J32}

\hypersetup{
  pdfauthor={Benoit Claudon, Patrick Graf, Henri Guenancia and Philipp Naumann},
  pdftitle={Kahler spaces with zero first Chern class},
  pdfkeywords={\kahler spaces, klt singularities, vanishing first Chern class, holonomy, Albanese map, fundamental groups, Kodaira problem, decomposition theorem},
  pdfstartview={Fit},
  pdfpagelayout={OneColumn},
  pdfpagemode={UseNone},
  linktoc=all,
  breaklinks,
  linkcolor=[RGB]{0 0 96},
  citecolor=[RGB]{96 0 0},
  urlcolor=[RGB]{0 96 0},
  colorlinks}

\begin{document}

\begin{abstract}
Let $X$ be a compact \kahler space with klt singularities and vanishing first Chern class.
We prove the Bochner principle for holomorphic tensors on the smooth locus of $X$: any such tensor is parallel with respect to the singular Ricci-flat metrics.
As a consequence, after a finite quasi-\'etale cover $X$ splits off a complex torus of the maximum possible dimension.
We then proceed to decompose the tangent sheaf of $X$ according to its holonomy representation.
In particular, we classify those $X$ which have strongly stable tangent sheaf: up to quasi-\'etale covers, these are either irreducible Calabi--Yau or irreducible holomorphic symplectic. As an application of these results, we show that if $X$ has dimension four, then it satisfies Campana's Abelianity Conjecture.
\end{abstract}

\maketitle

\begingroup
\hypersetup{linkcolor=black}
\tableofcontents
\endgroup

\section{Introduction}

Let $X$ be a compact \kahler manifold such that $\cc1X = 0 \in \HH2.X.\R.$.
The celebrated Beauville--Bogomolov Decomposition Theorem states that $X$ admits an \'etale cover that splits as a product of a complex torus, Calabi--Yau manifolds and irreducible holomorphic symplectic manifolds, \emph{cf.}~\cite{Bea83}.
The proof relies in a crucial way on Yau's solution of the Calabi conjecture~\cite{Yau78} that enables one to equip $X$ with a \kahler, Ricci-flat metric $\om$.
From there, one can use the powerful tools of differential geometry (de Rham's splitting theorem, Cheeger--Gromoll core theorem, Berger--Simons holonomy classification and Bieberbach theorem) to exhibit the sought cover.

Over the last few decades, the tremendous advances of the (algebraic) Minimal Model Program in birational geometry have highlighted the importance to understand and classify varieties with mild singularities. By variety, we mean here either a complex projective variety or a compact \kahler space.

It is in this context that a lot of attention has been drawn in recent years towards generalizing the Beauville--Bogomolov Decomposition Theorem to compact \kahler spaces $X$ with klt singularities and trivial first Chern class. A first important step in that direction was made by~\cite{EGZ} who generalized Yau's theorem and constructed singular \kahler--Einstein metrics $\omega$ on such varieties $X$. Unfortunately, these metrics viewed on $\Reg X$ are geodesically incomplete, preventing most of the classical results in differential geometry from applying. However, breakthroughs relying on the algebraic MMP (\emph{e.g.}~\cite{GKKP11, Xu14, GKP16}) and the theory of algebraic foliations~\cite{Dru16} allowed for a better understanding of projective varieties with klt singularities and trivial first Chern class. Down the line, this enabled~\cite{GGK} to compute the holonomy of the singular Ricci-flat metrics $\omega$ and shortly after, H\"oring and Peternell gave a proof of the decomposition theorem in the projective setting, based on an algebraic integrability result for vector bundles~\cite{HP}.

Unfortunately, these spectacular results leave the \kahler case of the decomposition theorem largely open, which may sound odd in light of the fact that the proof in the smooth case relies entirely on transcendental methods!
In this paper, we prove the \kahler version of several building blocks that underly the proof of the decomposition theorem in the projective case.

\subsection*{Bochner principle and Albanese map}

As a first step towards the above goal, we prove the so-called \emph{Bochner principle}.

\begin{bigthm}[Bochner principle, \cref{bochner}] \label{intro:bochner}
Let $(X, \omega_X)$ be a normal compact \kahler space with klt singularities such that $\cc1X = 0 \in \HH2.X.\R.$ and let $\omega \in \{ \omega_X \}$ be the singular Ricci-flat metric.
Let $p, q \ge 0$ be non-negative integers.
Then any holomorphic tensor
\[ \tau \in \HH0.\Reg X.\T X^{\otimes p} \otimes \Omega_X^{\otimes q}. \]
is parallel with respect to $\omega$ on $\Reg X$.
\end{bigthm}

Quite generally, a sacrilegious bottleneck of the theory in its current state is that if $X$ is klt, we have no way of comparing the fundamental group of the smooth locus $\pi_1(\Reg X)$ to the whole $\pi_1(X)$.
This is important because differential-geometric methods only control the former group.
The issue is that the algebraic results of~\cite{GKP16} about \emph{\gkp covers} are not yet available for complex spaces.
See \cref{rem:existence GKP cover} for an in-depth discussion.

In fact, a large part of the present work is devoted to finding ways to bypass the above-mentioned lamentable limitation of the literature.
This can already be seen in \cref{intro:bochner}: even though its statement is a straightforward generalization of~\cite[Theorem~A]{GGK}, the proof is quite different because it needs to avoid the use of \gkp covers.
We refer to \cref{comparison} for a more thorough comparison of the two results.

Our next theorem is about the Albanese map of \kahler spaces $X$ as above: $\alb_X$ is surjective and after a finite \'etale base change, it becomes globally trivial (\cref{alb splits}).
This generalizes~\cite[Theorem~8.3]{Kaw85} from the projective case, but the proof relies on the Bochner principle.
An important consequence is the existence of so-called \emph{torus covers}.

\begin{bigthm}[Torus covers, \cref{torus cover}] \label{intro:tc}
Let $X$ be a normal compact \kahler space with klt singularities such that $\cc1X = 0 \in \HH2.X.\R.$.
Then there exist normal compact \kahler spaces $T$ and $Z$ with canonical singularities together with a \qe cover $\gamma \from T \x Z \lto X$ such that:
\begin{itemize}
\item $T$ is a complex torus of dimension $\wt q(X)$.
\item The canonical sheaf of $Z$ is trivial, $\can Z \isom \O Z$.
\item The augmented irregularity of $Z$ vanishes, $\wt q(Z) = 0$.
\end{itemize}
\end{bigthm}

\noindent
For the definition of the augmented irregularity $\wt q(X)$, see \cref{dfn irreg}.
Also, we remark that the torus cover is essentially unique by \cref{uniqueness torus cover}.

Direct applications of \cref{intro:tc} include an alternative proof of the Abundance Conjecture in the setting of compact \kahler spaces with \emph{canonical} singularities and trivial first Chern class (\cref{abundance c1=0}) and a characterization of torus quotients in terms of the augmented irregularity (\cref{tq irreg}).
Concerning the further study of klt \kahler spaces with trivial first Chern class, \cref{intro:tc} enables us to reduce most questions to the case of canonical singularities, trivial canonical bundle and, most importantly, vanishing augmented irregularity.

\subsection*{Holonomy and the flat factor}

The natural next step is to understand the holonomy of the Ricci-flat metrics on $\Reg X$.
This leads to a decomposition of the tangent sheaf of $X$ that already reflects the conjectural Beauville--Bogomolov decomposition on an infinitesimal level.
The missing last steps are the algebraic integrability of the summands and a splitting theorem for such foliations.

\begin{bigthm}[Holonomy covers, p.~\pageref{proof theorem holonomy}] \label{intro:hol}
Let $(X, \omega_X)$ be as in \cref{intro:bochner}.
Then after replacing $X$ by a finite \qe cover, there exists a direct sum decomposition of the tangent sheaf of $X$,
\[ \T X = \sF \oplus \bigoplus_{k \in K} \sE_k, \]
where the reflexive sheaves $\sF$ and $\sE_k$ satisfy the following:
\begin{itemize}
\item The sheaves $\sF$ and $\sE_k$ are foliations with trivial determinant.
\item The sheaf $\sF\big|_{\Reg X}$ is flat. More precisely, it is given by a special unitary representation of $\pi_1(\Reg X)$.
\item Each factor $\sE_k\big|_{\Reg X}$ is parallel and has full holonomy group either $\mathrm{SU}(n_k)$ or $\mathrm{Sp}(n_k/2)$, with respect to the pullback of the singular Ricci-flat metric $\omega$.
Here, $n_k=\mathrm{rk}(\sE_k)$.
Moreover, $\sE_k$ is strongly stable with respect to any \kahler class.
\item After passing to a torus cover $\gamma \from T \x Z \to X$ as in \cref{intro:tc}, the tangent sheaf of $T$ becomes a direct summand of $\gamma^{[*]} \sF$.
If $X$ is locally algebraic, then we have equality $\mathrm{pr}_T^* \T T = \gamma^{[*]} \sF$.
\end{itemize}
\end{bigthm}

As far as the flat factor $\sF$ is concerned, \cref{intro:hol} only provides partial information.
We do however have a complete understanding of $X$ in the case where the other summands $\sE_k$ vanish.
In particular, in this case we know that $\mathrm{pr}_T^* \T T = \gamma^{[*]} \sF$ without making any algebraicity assumptions:

\begin{bigthm}[Characterization of torus quotients, \cref{local flat tangent} and \cref{flat tangent implies torus}] \label{intro:torusquotient}
Let $X$ be a normal complex space with klt singularities.
If $\T{\Reg X}$ is flat, then $X$ has only finite quotient singularities.

\noindent 
Moreover, if $X$ is compact and \kahler, then it is a quotient of a complex torus by a finite group acting freely in codimension one.
\end{bigthm}

\cref{intro:torusquotient} would follow directly from \cref{intro:hol} as soon as the \gkp covers of~\cite{GKP16} are available.
In our non-algebraic setting, we cannot rely on that result, but capitalizing on the classical fact that \emph{isolated} singularities are algebraic, we are able to prove a ``generic'' version of~\cite{GKP16}, \emph{cf.}~\cref{gkp max strata}.
Coupling this with the resolution of the (log canonical) Lipman--Zariski conjecture due to Kov\'acs and the second author~\cite{GK13} and Druel~\cite{DruelZL} independently, this weaker statement turns out to be sufficient to obtain \cref{intro:torusquotient}.

\subsection*{Varieties with strongly stable tangent sheaf}

Theorems~\labelcref{intro:bochner}--\labelcref{intro:torusquotient}, together with some standard representation theory, imply \cref{intro:structure} below.
This result characterizes the conjectural building blocks of compact \kahler spaces with klt singularities and trivial first Chern class via stability properties of their tangent sheaf.
We refer to \cref{section strongly stable variety} for the relevant definitions.

\begin{bigcor}[Spaces with strongly stable tangent sheaf, \cref{TX ss structure}] \label{intro:structure}
Let $X$ be as in \cref{intro:tc}.
If $\T X$ is strongly stable, then $X$ admits a \qe cover that is either an irreducible Calabi--Yau variety or an irreducible holomorphic symplectic variety.
\end{bigcor}

\cref{intro:structure} is the \kahler version of~\cite[Theorem~E]{GGK}.
Its proof bypasses the use of Druel's splitting result for flat summands~\cite{Dru16} as well as the existence of \gkp covers.

\subsection*{Fundamental groups}

According to Campana's Abelianity Conjecture~\cite[Conjecture~7.3]{Campana04}, the fundamental group of the regular locus $\Reg X$ of a compact \kahler space with klt singularities such that $\cc1X = 0$ should be virtually abelian, and finite if the augmented irregularity of $X$ vanishes.
Unfortunately, this result seems out of reach for the moment, even in the projective case.
It is yet crucial to control $\pi_1(\Reg X)$, for instance by relating it to $\pi_1(X)$, in order to trivialize flat subsheaves of the tangent sheaf or to clear the difference between restricted holonomy groups (which are classified by Berger--Simons) and holonomy groups (which govern the geometry of $X$ via the Bochner principle).

The fundamental group of the whole space $\pi_1(X)$ is in general easier to understand but it might be much smaller that $\pi_1(\Reg X)$, as the standard \cref{kummer} shows. Our results can be combined to techniques borrowed from \emph{e.g.} \cite[\S~\!13]{GGK} to derive finiteness properties of $\pi_1(X)$ in any dimension, \emph{cf.} Theorem~\ref{pi1 c1=0}. 
%

If we restrict our attention to low dimensions, we arrive at much stronger statements.

\begin{bigthm}[Fundamental groups in dimension four, \cref{pi1 dim four}] \label{intro:pi1 II}
Let $X$ be as in \cref{intro:tc} and of dimension $\le 4$.
Then:
\begin{itemize}
\item $\pi_1(X)$ is virtually abelian.
\item If $\wt q(X) = 0$, then $\pi_1(X)$ is finite.
\end{itemize}
\end{bigthm}

\subsection*{Comparison with other works}

All our results have previously been established in \cite{GGK} under the assumption that $X$ is {\it projective}. In {\it loc. cit.} the projectivity assumption is used in a crucial way to obtain the existence of maximally quasi-\'etale covers \cite{GKP16} and to rely on Druel's splitting result \cite{Dru16}. These results are currently unavailable in the analytic setting and, as a result, our proofs of \emph{e.g.}~\cref{intro:bochner} and \cref{intro:hol}
follow a quite different path and are in the end somewhat more natural, if not simpler. 

A few months after the present article had been uploaded to the arXiv, B.~Bakker, C.~Lehn and the third-named author posted the preprint~\cite{BGL} where they prove the Decomposition Theorem in full generality, \emph{i.e.}~for any compact \kahler space with klt singularities and trivial first Chern class.
Of course, the present paper was largely motivated by the Decomposition Theorem in the \kahler setting, but we would like to point out that several of the arguments in~\cite{BGL} crucially rely on the existence of holonomy and torus covers (Theorems B and C of the present article).
Furthermore, some of the results proven here are logically independent of the splitting theorem of~\cite{BGL} (\emph{e.g.}~Theorem A or Theorem D).
For these reasons, we think that the two papers nicely complement each other.


\subsection*{Acknowledgements}

H.G.~and B.C.~would like to thank St\'ephane Druel and Matei Toma for several enlightening discussions.
P.G.~and P.N.~would like to thank Mihai P\u aun and Thomas Peternell for sharing their insight with them. The authors would like to thank an anonymous referee for his/her careful reading and for valuable comments.

\section{Preliminaries} \label{sec:basic material}

In this section we gather some basic material to be used in the rest of the article.

\subsection{Global conventions}

Unless otherwise stated, complex spaces are assumed to be countable at infinity, separated, reduced and connected.
Algebraic varieties and schemes are always assumed to be defined over the complex numbers.

\subsection{General definitions}

As a courtesy to the reader, we recall the following standard definitions.

\begin{dfn}[Irregularity] \label{dfn irreg}
The \emph{irregularity} of a compact complex space $X$ is $q(X) \defn \hh1.Y.\O Y.$, where $Y \to X$ is any resolution of singularities.
The \emph{augmented irregularity} of $X$ is
\[ \wt q(X) \defn \max \Big\{ q \big( \wt X \big) \;\Big|\; \wt X \to X \text{ quasi-\'etale cover} \Big\} \in \N_0 \cup \{ \infty \}. \]
\end{dfn}

\begin{rem} \label{rem KS}
If $X$ has rational (\emph{e.g.}~klt) singularities, one has $q(X) = \hh1.X.\O X.$.
If additionally $X$ is \kahler, then it follows from~\cite[Corollary~1.8]{KS} that
\[ q(X) = \hh0.Y.\Omegap Y1. = \hh0.X.\Omegar X1.. \]
\end{rem}

\begin{dfn}[Flat sheaves] \label{flat sheaf}
Let $X$ be an irreducible and reduced complex space and let $\wt X \to X$ be its universal cover.
We say that a rank $r$ vector bundle $E \to X$ is \emph{flat} if there exists a linear representation $\rho \from \pi_1(X) \to \GL r\C$ such that $E$ is isomorphic to the bundle $\factor{\wt X \times \C^r}{\pi_1(X)} \to X$, where $\pi_1(X)$ acts diagonally.
\end{dfn}

\begin{dfn} \label{dfn loc alg}
Let $X$ be a complex space.
We say that $X$ is \emph{locally algebraic}, or that $X$ has \emph{algebraic singularities}, if there exists a euclidean open cover $\{ U_i \}_{i \in I}$ of $X$ such that for every $i \in I$, there is a quasi-projective scheme $Y_i$, an open subset $V_i \subset Y_i^{\mathrm{an}}$ and a biholomorphic map $\phi_i \from U_i \bij V_i$.
\end{dfn}

\begin{exm}
Here are some (non-)examples of algebraic singularities.
\begin{enumerate}
\item Every complex manifold, and more generally every complex space with quotient singularities, is locally algebraic~\cite{Cartan57}.
\item\label{example isolated singularities} Every complex space with only isolated singularities is locally algebraic by~\cite[Theorem~3.8]{Artin69} or~\cite[corollaire~1, \S 3, chapitre~II]{Tou68}.
\item If the complex space $X$ admits a \lt algebraic approximation, then it is locally algebraic (see \cite[\S2.4]{AlgApprox} for the notion involved).
\item Take a non-isotrivial family of elliptic curves over the unit disc $\Delta$.
Restrict it to the punctured disc $\Delta^*$ and pull it back along the universal cover $\Delta \to \Delta^*$.
Now take a cone over this family fibrewise.
We obtain a log canonical threefold with one-dimensional singular locus which is \emph{not} locally algebraic, because the $j$-function associated to the exceptional divisor is not algebraic.
\end{enumerate}
\end{exm}

\subsection{Coverings of complex spaces}

We consistently use the following notation.

\begin{dfn}[Covering maps] \label{def:cover}
A \emph{cover} or \emph{covering map} is a finite, surjective morphism $\gamma \from Y \to X$ of normal, connected complex spaces.
The covering map $\gamma$ is called \emph{Galois} if there exists a finite group $G \subset \Aut(Y)$ such that $Y \to X$ is isomorphic to the quotient map $Y \lto \factor YG$.
\end{dfn}

\begin{dfn}[Quasi-\'etale maps] \label{def:quasietale}
A morphism $\gamma \from Y \to X$ between normal complex spaces is called \emph{quasi-\'etale} if $\gamma$ is of relative dimension zero and \'etale in codimension one.
In other words, $\gamma$ is quasi-\'etale if $\dim Y = \dim X$ and if there exists a closed subset $Z \subset Y$ of codimension $\codim{Y}{Z} \ge 2$ such that $\gamma\big|_{Y \setminus Z} \from Y \setminus Z \to X$ is \'etale.
\end{dfn}

By purity of branch locus and the extension theorem of~\cite[Theorem~3.4]{DG94}, we get an equivalence of categories between the quasi-\'etale covers of $X$ and the \'etale covers of $\Reg X$.
We emphasize that with our definitions, an \'etale or quasi-\'etale \emph{cover} is automatically finite.

We will use several times the fact that taking Galois closure also works in the analytic context.
Compare~\cite[Lemma~7.4]{TorusQuotients}.

\begin{lem}[Galois closure] \label{lem:galois closure}
Let $\gamma \from Y \to X$ be a covering map between normal complex spaces.
Then there exists a Galois cover $g \from \wt Y \to Y$ (with $\wt Y$ normal) such that the composed map $f = \gamma \circ g \from \wt Y \to X$ is also Galois and we have an equality of branch loci $\mathrm{Br}(f) = \mathrm{Br}(\gamma)$.
In particular, if $\gamma$ is quasi-\'etale, then so is $f$.
\end{lem}

\begin{proof}
The morphism $\gamma$ is \'etale over a Zariski open set $X^\circ\subset X$ and then it corresponds to a finite index subgroup $H\subset \pi_1(X^\circ)$. This subgroup has a finite number of conjugates in $\pi_1(X^\circ)$ and the intersection of these conjugates is a normal finite index subgroup $H^\circ$ in $\pi_1(X^\circ)$. This subgroup $H^\circ$ gives rise to a Galois cover
\[ \wt Y^\circ\lto Y^\circ:=\gamma^{-1}(X^\circ)\stackrel{\gamma}{\lto} X^\circ. \]
The finite morphisms $\wt Y^\circ\lto X^\circ$ and $\wt Y^\circ\lto Y^\circ$  can be extended over the whole of $X$ and $Y$ to finite morphisms $f:\wt Y\lto X$ and $g:\wt Y\lto Y$~\cite[Theorem~3.4]{DG94}. The factorization $f=\gamma\circ g$ exists (by construction) over $X^\circ$ and it extends naturally. The equality of the branch loci is obvious.
\end{proof}

As a consequence, quotient singularities can be characterized in terms of smoothness of quasi-\'etale covers.

\begin{lem}[Quotient singularities] \label{lem:quotient sing}
A germ of normal singularity $(X,x)$ is a quotient singularity if and only if it admits a smooth quasi-\'etale cover.
\end{lem}

\begin{proof}
If $(X, x)$ is isomorphic to $\left( \factor{\C^n}G, 0 \right)$, we may assume that $G$ contains no quasi-reflections~\cite{ShephardTodd54,Chevalley55} and then the quotient map $\mathbb C^n \to \factor{\C^n}G$ is quasi-\'etale.

Conversely, if $\gamma \from Y \to X$ is quasi-\'etale with $Y$ smooth, then we can take a Galois closure as in \cref{lem:galois closure}, \emph{i.e.}~a map $g \from \wt Y \to Y$ with $f = \gamma \circ g \from \wt Y \to X$ Galois.
Since $f = \gamma \circ g$ is quasi-\'etale, the map $g$ does not branch over any divisor in $Y$ and since $Y$ is smooth we deduce that the map $g$ is \'etale by purity of the branch locus.
This implies that $\wt Y$ is smooth as well and hence $(X, x)$ is a quotient singularity~\cite{Cartan57}\footnote{The reference~\cite{Cartan57} is also available at \url{http://www.numdam.org/article/SHC_1953-1954__6__A12_0.pdf}.}.
\end{proof}

\subsection{Slope stability}

Let $X$ be a compact, normal \kahler space of dimension $n$ and let $\alpha\in \HH2.X.\R.$ be a \kahler class. We fix a resolution $\pi:\wh X\to X$ which is isomorphic over $X_{\rm reg}$. Finally, let $\sE$ be a torsion-free sheaf of rank $r$ on $X$, and let $\wh \sE :=\pi^*\sE/{\rm tor}$.

The slope of $\sE$ with respect to $\alpha$, denoted by $\mu_{\alpha}(\sE)$ is defined by $\mu_{\alpha}(\sE):= \frac 1r c_1(\wh \sE)\cdot (\pi^*\alpha)^{n-1}$. It is easy to check that the slope is independent of the chosen resolution. 

\begin{dfn}[Stability] \label{defn stability}
Let $\sE$ be a reflexive sheaf on $X$. We say that 
\begin{enumerate}
\item $\sE$ is \emph {stable} with respect to $\alpha$ if for any non-zero subsheaf $\sF\subset \sE$ of rank less than $r$, we have $\mu_\alpha(\sF)<\mu_\alpha(\sE)$. 
\item  $\sE$ is \emph {polystable} with respect to $\alpha$ if it is a direct sum of stable sheaves with identical slope.
\item $\sE$ is \emph {strongly stable} with respect to $\alpha$ if for any quasi-\'etale cover $f:Y\to X$, the sheaf $f^{[*]}\sE:=(f^*\sE)^{**}$ is stable with respect to $f^*\alpha$. 

\end{enumerate}
\end{dfn}

\subsection{Stratifications of complex spaces}

Here we just set up some notation to be used when dealing with stratifications of singular spaces.
We first recall that a \emph{Whitney stratification} of a closed subspace $X$ of a smooth real manifold $M$ is a locally finite partition $\sS$ of $X$ into smooth, connected, locally closed subsets $(X_\lambda)_{\lambda\in \Lambda}$:
\[ (\sS): \quad X = \bigsqcup_{\lambda\in\Lambda} X_\lambda \]
satisfying certain incidence conditions.
The pair $(X,\sS)$ is called a \emph{stratified space} and the $X_\lambda$ are called the \emph{strata} of $(X, \sS)$.
We refer to~\cite[Part~I, \S~1.2 and~1.3]{GMbook} or to~\cite{Mather70} for more background on this classical topic, in particular the existence of Whitney stratifications.

We will mostly be interested in the situation where $X$ is an irreducible and reduced complex space.
In this case, the strata are asked to be smooth and connected locally closed \emph{analytic} subsets.
In particular, by dimension we will always mean the dimension as a complex space.

A stratification can alternatively be encapsulated in a filtration of $X$ by closed analytic subsets:
\[ S_k \defn \underset{\{\lambda\,\mid\, \dim(X_\lambda)\le k\}}{\bigsqcup} X_\lambda. \]
So defined, $S_k$ has dimension at most $k$ and if $n \defn \dim X$ we have:
\[ S_0 \subset S_1 \subset \cdots \subset S_n = X. \]
We will be mainly interested in what we call \emph{maximal strata}.

\begin{dfn}[Maximal strata] \label{def:max strata}
Let $\sS:=(X_\lambda)_{\lambda\in \Lambda}$ be a Whitney stratification of $X$ and $(S_k)_{0\le k\le n}$ be the corresponding filtration.
We let $d < n$ be the maximum index such that $S_d \subsetneq X$.
The strata indexed by
\[ \Lambda_\maxi := \left\{\lambda\in\Lambda\mid \dim(X_\lambda)=d\right\} \]
are called the \emph{maximal (dimensional) strata}.
\end{dfn}

This definition is meaningful in particular when the stratification $\sS$ satisfies $S_d = \Sing X$.

\begin{dfn}[Refining stratifications] \label{refine}
Let $\sS$ and $\sS'$ be two Whitney stratifications of $X$.
We say that $\sS'$ is a \emph{refinement} of $\sS$ if the strata of $\sS$ are unions of strata of $\sS'$.
In this case we write $\sS' < \sS$.
\end{dfn}

Since strata are irreducible by definition, in the above situation each $\sS$-stratum contains exactly one $\sS'$-stratum as an open subset.

\begin{dfn}[Stratified maps]
A continous map $f \from X \to Y$ is called \emph{stratified} (with respect to Whitney stratifications $\sS_X$ and $\sS_Y$ of $X$ and $Y$, respectively) if for each $\sS_Y$-stratum $A \subset Y$, the preimage $f\inv(A)$ is a union of $\sS_X$-strata and $f$~takes each of these strata smoothly and submersively to $A$.
\end{dfn}

\begin{thm}[Stratifications of maps, \emph{cf.}~\protect{\cite[Part~I, \S1.7]{GMbook}}] \label{strat map}
Let $f \from X \to Y$ be a holomorphic map, $X$ and $Y$ being endowed with Whitney stratifications as above.
Then there exist refinements $\sS_X' < \sS_X$ and $\sS_Y' < \sS_Y$ such that $f \from (X, \sS_X') \to (Y, \sS_Y')$ is a stratified map. \qed
\end{thm}

We finally recall that the transverse structure of a given stratum is topologically locally trivial in a rather strong sense.
In the following statement, a \emph{normal (or transverse) slice} $N_\lambda$ to a stratum $X_\lambda$ is a smooth submanifold of the ambient space $M$ which is transverse to each stratum of $X$, intersects $X_\lambda$ in a single point and satisfies $\dim(N_\lambda) = \codim{M}{X_\lambda}$. We denote by $\bS^k$ the euclidean unit sphere in $\mathbb R^{k+1}$.

\begin{thm}[\emph{cf.} Part~I, \S~\!\!1.4 in~\cite{GMbook}] \label{thm:neighborhood stratum}
Let $X_\lambda$ be any stratum of $(X,\sS)$.
Then there exists a stratified set $L\subset \bS^{k}$ for some integer $k$ such that any point $x_0\in X_\lambda$ has a neighborhood $U$ in $X$ endowed with a homeomorphism
\[ h \from U \longrightarrow (U\cap X_\lambda)\times c(L) \]
where $c(L)$ is the cone over $L$. The space $L$ is called the \emph{link} of the stratum $X_\lambda$ and is denoted $\mathrm{Link}(X,X_\lambda)$.

The space $c(L)$ is endowed with a natural stratification whose strata are the vertex of the cone and subspaces of the form $A \x ]0, 1]$, with $A$ a stratum of $L$.
The homeomorphism $h$ sends the strata of $U$ into those of $(U \cap X_\lambda) \x c(L)$.
In particular, if $N$ is any transverse slice to $X_\lambda$ at $x_0$ then $(U \cap N) \setminus \set{x_0}$ has the homotopy type of the link $\mathrm{Link}(X, X_\lambda)$. \qed
\end{thm}

This local picture can be turned into a global one.

\begin{thm}[Existence of tubular neighborhoods,~\cite{Mather70}] \label{thm:tubular neighborhood}
Let $X_\lambda$ be any stratum of $(X, \sS)$.
Then there exists a closed neighborhood $U_\lambda$ of $X_\lambda$ in $X$ and a continuous retraction
\[ f_\lambda \from U_\lambda \lto X_\lambda\]
that is locally topologically trivial with fibre $c(\mathrm{Link}(X,X_\lambda))$.
The natural inclusion $X_\lambda \inj U_\lambda$ being a section of $f_\lambda$, the punctured open neighborhood $U_\lambda^\circ \setminus X_\lambda$ is also locally topologically trivial over $X_\lambda$ with fibre homotopy equivalent to $\mathrm{Link}(X,X_\lambda)$.
\end{thm}

This topological stability will be used through the following corollary.

\begin{cor} \label{cor:branching stratum}
Let $\gamma \from Y \to X$ be a \qe cover between normal complex spaces, and let $\sS$ be a Whitney stratification of $X$ such that $\Sing X$ is a union of strata.
Then for any maximal stratum $X_\lambda$, the following dichotomy holds: either $X_\lambda \subset \mathrm{Br}(\gamma)$ or $X_\lambda \cap \mathrm{Br}(\gamma) = \emptyset$.
\end{cor}

\begin{proof}
This is a direct application of~\cite[Corollary~3.12]{GKP16}, where $A = \Sing X$.
Let $\lambda \in \Lambda_\maxi$ be any index.
If $X_\lambda \subset A$, then $\overline{X_\lambda}$ equals an irreducible component of $A$ by maximality.
Hence we may apply said result.
Otherwise, $X_\lambda$ must be disjoint from $A$, \emph{i.e.}~contained in $\Reg X$.
So we are in the second case of the dichotomy.
\end{proof}

\begin{rem} \label{rem:canonical stratification}
The notion of stratum and thus of its normal topological type depends on a given stratification.
We will not need it in the sequel, but it has to be noted that any irreducible reduced complex space $X$ has a canonical Whitney stratification $\sS_{\mathrm{can}}$.
This means that if $\sS$ is any other Whitney stratification of $X$, the strata of $\sS_{\mathrm{can}}$ are unions of strata of $\sS$.
The construction of $\sS_{\mathrm{can}}$ is explained in~\cite[Theorem~4.9]{Mather73} (see also~\cite[corollaire~6.1.7]{LeTeissier} where the canonical stratification is defined in terms of polar varieties).
\end{rem}

\section{Bochner principle for reflexive tensors} \label{sec bochner}

In this section, we prove the Bochner principle, \cref{intro:bochner}.
For generalities about \kahler metrics or currents on normal complex spaces, we refer to~\cite[A.1]{Dem85}, \cite[\S~\!5.2]{EGZ}, \cite[\S~\!4.6.1]{BEG} or~\cite[\S~\!3]{TorusQuotients}.
We work in the following setting.

\begin{setup} \label{setup1}
Let $X$ be an $n$-dimensional complex, compact \kahler space with klt singularities such that $K_X$ is numerically trivial. 

We choose a \kahler resolution $\pi \from Y \to X$ such that $\Exc(\pi)$ is an snc divisor $F = \sum_j F_j$, and we fix holomorphic sections $s_j \in \HH0.Y.\O Y(F_j).$ cutting out the smooth component $F_j$.
We set $s_F \defn \otimes_j s_j$ and $Y^\circ \defn Y \setminus F$.

Fix two integers $p, q \ge 0$.
We set $\sE_X \defn \big( \T X^{\tensor p} \tensor \Omega_X^{\tensor q} \big) \ddual$, where $(-) \ddual$ denotes the double dual, and $\sE_Y \defn \T Y^{\tensor p} \tensor \Omega_Y^{\tensor q}$.
One has a natural inclusion of coherent sheaves $\pi_* \sE_Y \subset \sE_X$ and the quotient sheaf $\factor{\sE_X}{\pi_* \sE_Y}$ is torsion, being supported on~$\Sing X$.
Thanks to R\"uckert's Nullstellensatz~\cite[Ch.~3, \S2]{CAS}, there exists an integer $k \ge 1$ such that for any reflexive tensor $\tau \in \HH0.X.\sE_X.$, the section $\pi\big|_{Y^\circ}^* \big( \tau\big|_{\Reg X} \big)$ extends to an element $\sigma \in \HH0.Y.\sE_Y \otimes \O Y(kF).$.

Thanks to the combination of~\cite[Theorem~7.5]{EGZ} and~\cite[Corollary~1.1]{Paun}, there exists in each \kahler class $\alpha\in \HH2.X.\R.$ a unique closed, positive current $\omega_\alpha \in \alpha$ with bounded potentials, smooth on $\Reg X$ and satisfying $\Ric \omega_\alpha  = 0$ on this locus, \emph{cf.}~\cref{rem KE} below. 

The \kahler metric $\omega_\alpha$ on $\Reg X$ induces a smooth, hermitian metric on $\T{\Reg X}$ and in turn on $\sE_X\big|_{\Reg X}$, with Chern connection~$D_{\sE_X}$.
\end{setup}

\begin{rem} \label{rem KE}
Given a \kahler metric $\om_X$ on $X$, it follows from  \cite[Theorem~7.5]{EGZ} that there exists a unique \kahler--Einstein current $\om_\alpha$ in the cohomology class $\{\omega_X\} \in \HH1.X.\mathrm{PH}_X.$ under the connecting map $\HH0.X.\factor{L^{\infty}_X}{\mathrm{PH}_X}. \lto \HH1.X.\mathrm{PH}_X.$. This cohomology group identifies two $(1,1)$-currents with local (bounded) $dd^c$-potentials if and only if they differ by the $dd^c$ of a global (bounded) function.
It is however more convenient to view $\om_\alpha$ in the more familiar cohomology space $\HH2.X.\R.$ using the connecting map arising from the natural exact sequence
\[ 0 \longrightarrow \underline{\R}_X \longrightarrow \O X \xrightarrow{\;\mathrm{Im}\;} \mathrm{PH}_X \longrightarrow 0. \]
It was recently proved in~\cite[Proposition~3.5]{TorusQuotients} that the map $\HH1.X.\mathrm{PH}_X. \to \HH2.X.\R.$ is injective whenever $X$ has rational singularities.
Therefore, a \kahler class $\alpha \in \HH2.X.\R.$ is associated to a unique class of a \kahler metric $\{\omega_X\} \in \HH1.X.\mathrm{PH}_X.$.
This allows us to consider without any ambiguity the \emph{unique} singular Ricci-flat metric $\omega_\alpha$ in $\alpha \in \HH2.X.\R.$ as stated a few lines above. 
\end{rem}

\begin{rem} \label{abundance}
In \cref{setup1} above, the Abundance Conjecture has been proved in~\cite[Corollary~1.18]{JHM2} as an application of results of B.~Wang~\cite{WangB}.
So there exists a \qe cover $X' \to X$ such that $K_{X'}$ is trivial and, in particular, $X'$ has canonical singularities.
We do not rely on this fact to prove the Bochner principle.

Quite the opposite is true: one may use the Bochner principle to prove Abundance in \cref{setup1}, at least under the slightly stronger assumption that $X$ has canonical (as opposed to klt) singularities.
The details of the argument can be found in \cref{abundance c1=0}.
\end{rem}

\begin{thm}[Bochner principle] \label{bochner}
Let $X$ be a compact \kahler space with klt singularities such that $K_X$ is numerically trivial and let $\alpha$ be a \kahler class.
With the notation of \cref{setup1} above, any holomorphic tensor $\tau \in \HH0.\Reg X.\sE_X.$ is parallel with respect to $\omega_\alpha $, \emph{i.e.}~$D_{\sE_X} \tau = 0$ on $\Reg X$.
\end{thm}

\begin{rem}[Comparison with earlier results] \label{comparison}
In~\cite{GGK}, the Bochner principle is proved under the assumption that $X$ is projective.
The proof goes as follows: along the lines of~\cite{GSS}, one can prove the Bochner principle for \emph{bundles}~\cite[Theorem~8.1]{GGK}, stating that a saturated subsheaf of slope zero of a tensor bundle $\sF \subset \sE_X$ is automatically parallel with respect to the \kahler--Einstein metric on the smooth locus.
This does not rely on the projectivity assumption.

To go from bundles to tensors, one considers the line bundle generated by a given tensor $\tau$ on $\Reg X$.
As it is parallel, the holonomy group $G$ acts on it, yielding a character of $G$.
The classification of the holonomy~\cite[Theorem~B]{GGK} shows that up to passing to a \qe cover, $G$ is semi-simple, hence the character is trivial and $\tau$ itself is parallel.
The above classification result relies on the projectivity assumption through the existence of a \gkp cover~\cite[Theorem~1.5]{GKP16} and Druel's integrability result~\cite[Theorem~1.4]{Dru16}.
\end{rem}

\begin{proof}[Proof of \cref{bochner}]
We first need to set up some notation. 

\noindent
\textbf{Hermitian metrics. }
We choose some smooth hermitian metrics $h_j$ on $\O Y(F_j)$. We denote by $\theta_{j}$ the curvature form of that metric, i.e $\theta_j:=i\Theta_{h_{j}}(F_j)$.
Next, we define for any $\eps > 0$ another smooth metric $h_{j,\eps}$ on $\O Y(F_j)$ by 
\[ h_{j,\eps}:=\frac{1}{|s_j|_{h_j}^2+\eps^2}\cdot h_j. \]
The curvature form $\theta_{j,\eps}$ of that metric, \emph{i.e.} $\theta_{j,\eps}:=i\Theta_{h_{j,\eps}}(F_j)$, is given by 
\[ \theta_{j,\eps}= \underbrace{\frac{\eps^2  }{(|s_j|^2+\eps^2)^2}\cdot Ds_j\wedge \overline {D s_j}}_{=:\beta_{j,\eps}}+\underbrace{\frac{\eps^2}{|s_i|^2+\eps^2}\cdot \theta_j}_{=:\gamma_{j,\eps}}. \]
The symbol $D$ above refers to the Chern connection induced by $h_j$ on $\O Y(F_j)$. Finally, we set 
\[ h_F:=\prod_j h_j \quad \mbox{and} \quad h_{F,\eps}=\prod_j h_{j,\eps}=\frac{1}{\prod_j (|s_j|^2+\eps^2)}\cdot h_F; \] they define smooth metrics on $\O Y(F)$.\\

\noindent
\textbf{Approximate \kahler-Einstein metrics. }
We introduce the rational coefficients $a_i> -1$ such that $K_Y=\pi^*K_X+\sum a_i F_i$. We fix a \kahler reference metric $\omega_Y$ on $Y$, and consider, for each $\eps,t>0$, the unique \kahler metric $\omega_{t,\eps}\in \pi^*\alpha+t\{\omega_Y\}$ solution of
\[ \Ric \omega_{t,\eps} = -\sum_j a_j \theta_{j,\eps} \]
Its existence is guaranteed by Yau's solution of the Calabi conjecture~\cite{Yau78}.
In terms of Monge-Amp\`ere equations, if $\omega_X$ is a smooth representative of $\alpha$, then $\omte=\pi^*\om_X+t\om_Y+\ddc \vp_{t,\eps}$ is solution of \[ (\pi^*\om_X+t\om_Y+\ddc \vp_{t,\eps})^n = \frac{e^{-c_{t,\eps}}dV}{\prod (|s_j|^2+\eps^2)^{a_j}} \]
where $dV$ is a fixed smooth volume form such that $\Ric dV= \sum a_j \theta_j$ and $c_{t,\eps}\in \mathbb R$ is a normalizing constant defined by $e^{c_{t,\eps}}=\frac{1}{(\pi^*\alpha+t\{\om_Y\})^n}\int_Y  \frac{dV}{\prod (|s_j|^2+\eps^2)^{a_j}}$.\\

\noindent
\textbf{Curvature formula. }
We set $\sE_Y(kF):=\sE_Y \otimes \O Y(kF)$ and we choose $\sigma \in \HH0.Y.\sE_Y(kF).$ some meromorphic tensor on $Y$. The \kahler metric $\om_{t,\eps}$ induces a smooth hermitian metric $h_{\omega_{t,\eps}}$ on $\sE_Y$. We consider the metric 
\[ h_{t,\eps}:=h_{\omega_{t,\eps}}\otimes h_{F,\eps}^{\otimes k} \quad \mbox{on } \sE_Y(kF) \]
with Chern connection $D$ and set $|\sigma|:=|\sigma|_{h_{t,\eps}}$. 
We have the following Poincar\'e-Lelong type formula

\newcommand{\ud}{|\sigma|^2+1}

\begin{equation} \label{ineq0}
\ddc \log(\ud) = \frac{1}{\ud}\left( |D\sigma|^2-\frac{|\langle D\sigma,\sigma\rangle |^2}{\ud}-\la i\Theta_{h_{t,\eps}}(\sE_Y(kF))\sigma,\sigma\ra \right) 
\end{equation}
Wedging this last inequality with $\om_{t,\eps}^{n-1}$ and integrating it on $X$ yields:
\[ \int_Y  \frac{\la i\Theta_{h_{t,\eps}}(\sE_Y(kF))\sigma,\sigma\ra}{\ud} \wedge \om_{t,\eps}^{n-1} =  \int_Y \frac{1}{\ud}\left( |D\sigma|^2-\frac{|\la D\sigma,\sigma\ra |^2}{\ud} \right)  \wedge \om_{t,\eps}^{n-1} \]
As $|\la D\sigma,\sigma\ra | \le |D \sigma| \cdotp |\sigma|$, we obtain
\begin{equation}
\label{ineq00}
\int_Y  \frac{\la i\Theta_{h_{t,\eps}}(\sE_Y(kF))\sigma,\sigma\ra}{\ud} \wedge \om_{t,\eps}^{n-1} \ge  \int_Y \frac{ |D\sigma|^2}{(\ud)^2} \wedge \om_{t,\eps}^{n-1}
\end{equation}
Now, one has
\[ i\Theta_{h_{t,\eps}}(\sE_Y(kF))=i\Theta_{h_{\omega_{t,\eps}}}(\sE_Y)\otimes \mathrm{Id}_{\O Y(kF)} + \mathrm{Id}_{\sE_Y}\otimes i\Theta_ {h_{F,\eps}^{\otimes k}}(\O Y(kF)). \]

\newcommand{\rr}{^{\otimes p}}

First let us introduce a notation: let $V$ be a complex vector space of dimension $n$, let $p\ge 1$ be an integer, and let $f\in \mathrm{End}(V)$. We denote by $f\rr$ the endomorphism of $V^{\otimes p}$ defined by \[f\rr(v_1 \otimes \cdots \otimes v_p) := \sum_{i=1}^p v_1 \otimes \cdots v_{i-1} \otimes f(v_i) \otimes v_{i+1} \otimes \cdots \otimes v_p\]
Let us add that if $V$ has an hermitian structure and if $f$ is hermitian semipositive, then so is $f\rr$ with the induced metric, and we have $\tr (f\rr) = p n^{p-1} \tr(f)$.
Now we can easily check the following identity:
\[n  i\Theta( T_Y\rr,h_{\om_{t,\eps}})  \wedge \om_{t,\eps}^ {n-1} = (\sharp \Ric \om_{t,\eps})\rr \,\om_{t,\eps}^ n \]
where $\sharp \Ric \om$ is the endomorphism of $T_Y$ induced by $\Ric \om_{t,\eps}$ via $\om_{t,\eps}$. As 
 $\Ric \om =-\sum a_j\theta_{j,\eps}$, we deduce that 
\[ \tr_{\omte} i\Theta_{h_{\omega_{t,\eps}}}(\sE_Y)  = -\sum_j a_j\left[ (\sharp \theta_{j,\eps})\rr \otimes \mathrm{Id}_{{T_Y^*}^{\otimes q}}-   \mathrm{Id}_{{T_Y}^{\otimes p}} \otimes \overline{ (\sharp \theta_{j,\eps})}^{\otimes q}  \right] \]
 while
\[ i\Theta_ {h_{F,\eps}^{\otimes k}}(\O Y(kF)) \wedge \om_{t,\eps}^{n-1}=k\sum_j \theta_{j,\eps}\wedge \om_{t,\eps}^{n-1}. \]
Locally, one can choose a trivialization $e$ of $\O Y(kF)$ and write $\sigma = u \otimes e$ where $u$ is a local section of $\sE_Y$. Then  
\begin{equation}
\label{part0}
\la i\Theta_{h_{t,\eps}}(\sE_Y(kF))\sigma,\sigma\ra =\underbrace{\la i\Theta_{h_{\om_{t,\eps}}}(\sE_Y)u,u\ra |e|_{h_{F,\eps}^{\otimes k}}^2}_{\text{(I)}}+\underbrace{(k\sum_j \theta_{j,\eps})\cdot |\sigma|^2}_{\text{(II)}}\\
\end{equation}

\noindent
\textbf{Computation of Term (I).}\\
This part is entirely similar to~\cite[\S 9]{GGK}.
We have the decomposition 
\[ (\sharp \theta_{j,\eps})\rr  = (\sharp \beta_{j,\eps})\rr+(\sharp \gamma_{j,\eps})\rr \]
along with the inequalities 
\begin{equation}
\label{beta}
0 \le  (\sharp \beta_{j,\eps})\rr  \le \left(\tr_{\rm End} (\sharp \beta_{j,\eps})\rr \right) \mathrm{Id}_{T_Y^{\otimes p}}  
   \le pn^{p-1} \tr_{\om_{t,\eps}} \beta_{j,\eps} \cdot  \mathrm{Id}_{T_Y^{\otimes p}} 
\end{equation}
and 
\begin{equation}
\label{gamma}
 \pm (\sharp \gamma_{j,\eps})\rr \le  \frac{C\eps^2}{|s_j|^2+\eps^2} \cdot (\sharp \om_Y)\rr \le Cpn^{p-1} \cdot \frac{\eps^2}{|s_j|^2+\eps^2} \cdot \tr_{\om_{t,\eps}}  \om_Y\cdot  \mathrm{Id}_{T_Y^{\otimes p}}
\end{equation}
as soon as $C>0$ is large enough so that $\pm \theta_j \le C \om_Y$ for any $j$. 

\noindent
From~\labelcref{beta}, one deduces
\begin{align*}
0 \le \frac{ \la  (\sharp \beta_{j,\eps})\rr \otimes  \mathrm{Id}_{{T_Y^*}^{\otimes q}} u,u\ra |e|_{h_{F,\eps}^{\otimes k}}^2}{|\sigma|^2+1} &\le  pn^{p-1}\tr_{\om_{t,\eps}} \beta_{j,\eps} \frac{|\sigma|^2}{|\sigma|^2+1}\\
& \le  pn^{p-1}\tr_{\om_{t,\eps}} \beta_{j,\eps} 
\end{align*}
while from~\labelcref{gamma}, one deduces
\begin{align*}
\pm \frac{ \la  (\sharp \gamma_{j,\eps})\rr \otimes  \mathrm{Id}_{{T_Y^*}^{\otimes q}} u,u\ra |e|_{h_{F,\eps}^{\otimes k}}^2}{|\sigma|^2+1} &\le   Cpn^{p-1} \cdot \frac{\eps^2}{|s_j|^2+\eps^2} \cdot \tr_{\om_{t,\eps}}  \om_Y \frac{|\sigma|^2}{|\sigma|^2+1}\\
& \le   Cpn^{p-1} \cdot \frac{\eps^2}{|s_j|^2+\eps^2} \cdot \tr_{\om_{t,\eps}}  \om_Y
\end{align*}
so in conclusion, one gets
\begin{equation*}
\pm \frac{ \la  (\sharp \theta_{j,\eps})\rr \otimes  \mathrm{Id}_{{T_Y^*}^{\otimes q}} u,u\ra |e|_{h_{F,\eps}^{\otimes k}}^2}{|\sigma|^2+1} \le C\left[\tr_{\om_{t,\eps}} \beta_{j,\eps} +\frac{\eps^2}{|s_j|^2+\eps^2} \cdot \tr_{\om_{t,\eps}}  \om_Y\right]
\end{equation*}
for some $C$ large enough. Performing the same computations with $ \mathrm{Id}_{{T_Y}^{\otimes p}} \otimes \overline{ (\sharp \theta_{j,\eps})}^{\otimes q}$, one eventually obtains 
\begin{equation}
\label{part1}
\pm \frac{\la i\Theta_{h_{\om_{t,\eps}}}(\sE_Y)u,u\ra |e|_{h_{F,\eps}^{\otimes k}}^2}{|\sigma|^2+1} \wedge \om_{t,\eps}^{n-1} \le C\left[(\beta_{j,\eps} +\frac{\eps^2}{|s_j|^2+\eps^2}\cdot \om_Y)\wedge \om_{t,\eps}^{n-1}\right]
\end{equation}

\bigskip

\noindent
\textbf{Computation of Term (II).}

\noindent
By the same token as before, one can decompose
\[ \theta_{j,\eps}\cdot \frac{|\sigma|^2}{|\sigma|^2+1}\wedge \om_{t,\eps}^{n-1}=  \underbrace{\frac{|\sigma|^2}{|\sigma|^2+1}\cdot  \beta_{j,\eps}\wedge \om_{t,\eps}^{n-1}}_{\text{(III)}}+ \underbrace{\frac{|\sigma|^2}{|\sigma|^2+1}\cdot \gamma_{j,\eps}\wedge \om_{t,\eps}^{n-1}}_{\text{(IV)}} \]
and write $0 \le \text{(III)} \le \beta_{j,\eps} \wedge \om_{t,\eps}^{n-1}$ and $\pm \text{(IV)} \le \frac{C\eps^2}{|s_j|^2+\eps^2}\cdot \om_Y\wedge \om_{t,\eps}^{n-1}$
to obtain
\begin{equation}
\label{part2}
  \theta_{j,\eps}\cdot \frac{|\sigma|^2}{|\sigma|^2+1}\wedge \om_{t,\eps}^{n-1}\le C\left[(\beta_{j,\eps} +\frac{\eps^2}{|s_j|^2+\eps^2}\cdot \om_Y)\wedge \om_{t,\eps}^{n-1}\right]
\end{equation}

\noindent
Putting~\labelcref{part1} and~\labelcref{part2} together, one sees from~\labelcref{part0} that
\begin{equation}
\label{final}
\pm \frac{\la i\Theta_{h_{t,\eps}}(\sE_Y(kF))\sigma,\sigma\ra}{|\sigma|^2+1}\wedge \om_{t,\eps}^{n-1}\le C\left[\sum_j \Big(\beta_{j,\eps} +\frac{\eps^2}{|s_j|^2+\eps^2}\cdot \om_Y\Big)\wedge \om_{t,\eps}^{n-1}\right]
\end{equation}
Finally, one has 
\begin{align*}
\int_Y \beta_{j,\eps} \wedge \om_{t,\eps}^{n-1} &=\int_Y \theta_{j,\eps} \wedge \om_{t,\eps}^{n-1}- \int_Y \gamma_{j,\eps} \wedge \om_{t,\eps}^{n-1} \\
& \le F_j \cdot (\pi^*\alpha+t{\om_Y})^{n-1} + C \int_{Y}\frac{\eps^ 2}{|s_j|^2+\eps^ 2} \, \om_Y \wedge \om_{t,\eps}^ {n-1}
\end{align*}
so that~\labelcref{final} combined with \cite[Claim~9.5]{GGK} show that 
\begin{equation}
\label{limit}
\varlimsup_{t\to 0} \varlimsup_{\eps \to 0}\int_Y \frac{\la i\Theta_{h_{t,\eps}}(\sE_Y(kF))\sigma,\sigma\ra}{|\sigma|^2+1}\wedge \om_{t,\eps}^{n-1} = 0.
\end{equation}

\bigskip

\noindent
\textbf{Conclusion.}

\noindent
When $\eps,t$ go to zero, $\om_{t,\eps}$ converge weakly to $\pi^*\omega_\alpha $, and the convergence is smooth on $Y^\circ$. Inequality~\labelcref{ineq00} combined with Fatou lemma ensure that $\sigma$ has zero covariant derivative on $Y^\circ$ with respect to the smooth hermitian metric $h_{\pi^*\omega_\alpha } \otimes h_{F,0}^{\otimes k}$ on $\sE_Y(kF)|_{Y^\circ}$. Now, it follows from the definition of $h_{F,0}$ that $|s_F|_{h_{F,0}}\equiv 1$ on $Y$ so that $s_F$ is parallel with respect to $h_{F,0}$ on $Y^{\circ}$. This implies that $\sigma/s_F^{\otimes k}$ is parallel with respect to $h_{\pi^*\omega_\alpha }$ on $Y^\circ$, hence $\tau$ is parallel with respect to $h_{\omega_\alpha}$ on $\Reg X$. 
\end{proof}

%

\begin{rem} \label{poles}
In \cref{bochner} above, if $\sE_X = \T X$ or $\sE_X = \Omegar Xp$ for some $1 \le p \le n$, then one can choose the resolution $\pi$ such that the tensor $\tau$ pulls back to a holomorphic tensor on the resolution (\emph{i.e.}~it does not acquire any poles along $F$).
This follows from the existence of a functorial resolution of singularities in the first case (see \emph{e.g.}~\cite[Proposition~3.9.1]{Kol07}) and from~\cite{KS} in the second one.
In such a case, the term $\text{(II)}$ in the proof above disappears and the proof reduces to the computations of~\cite{GSS}.
We do, however, need the Bochner principle also for more general tensors.
\end{rem}

\section{Structure of the Albanese map} \label{section:albanese}

For generalities on the Albanese map for singular spaces, see~\cite[\S 3]{AlgApprox}.
We prove here the exact analogue of~\cite[Theorem~8.3]{Kaw85} in the \kahler setting.
See also~\cite[Theorem~1.10]{AlgApprox} for the three-dimensional case.
We then proceed to prove the existence of torus covers (\cref{intro:tc}) and some immediate corollaries.

\begin{thm}[Albanese splits after base change] \label{alb splits}
Let $X$ be a normal compact \kahler space with canonical singularities such that $\cc1X = 0$.
Let
\[ \alpha \from X \lto A \defn \Alb(X) \]
be the Albanese map of $X$.
Then there exists a finite \'etale cover $A_1 \to A$ such that $X \x_A A_1$ is isomorphic to $F \x A_1$ over $A_1$, where $F$ is connected.
In particular, $\alpha$ is a surjective analytic fibre bundle with connected fibres, and $q(X) \le \dim X$.
\end{thm}

\begin{rem-plain}
Under the stronger assumption that $K_X$ is torsion, the fact that $\alpha$ is surjective with connected fibres has already been shown in~\cite[Theorem~24]{Kaw81} and~\cite[Theorem~5.1 and Proposition~5.3]{Campana04}.
Since these results are stated for manifolds, one needs to apply them to a smooth model $Y \to X$.
This is possible since $K_Y$ is a \Q-effective exceptional divisor and so $Y$ has vanishing Kodaira dimension.

On the other hand, we can use \cref{alb splits} to prove that $K_X$ is torsion under the assumptions of that theorem.
See \cref{abundance c1=0}.
\end{rem-plain}

\begin{proof}[Proof of \cref{alb splits}]
By the universal property of the Albanese torus, for every $g \in \Aut(X)$ there exists a unique $\phi_g\!: A \to A$ such that $\phi_g \circ \alpha = \alpha \circ g$.
If $g \in \Autn(X)$, the identity component, then $g$ acts trivially on $\HH1.X.\C.$ and in particular on $\HH0.X.\Omegar X1.\dual$.
Consequently, the linear part of $\phi_g$ is the identity.
That is, $\phi_g$ is a translation by some element $\phi(g) \in A$.
We have thus defined a homomorphism of complex Lie groups $\phi \from \Autn(X) \to A$, which we will now show to be the desired cover $A_1 \to A$.
First of all, since $X$ has canonical singularities, it is not uniruled.
Hence $\Autn(X)$ is a complex torus by~\cite[Proposition~5.10]{Fuj78}.
It is thus sufficient to show that the induced Lie algebra map $\d\phi \from \HH0.X.\T X. \to \HH0.X.\Omegar X1.\dual$ is an isomorphism.
One checks easily that $\d\phi$ is given by the natural contraction pairing
\begin{equation} \label{contraction}
\HH0.X.\T X. \x \HH0.X.\Omegar X1. \lto \HH0.X.\O X. = \C.
\end{equation}
Fix a \kahler class $a \in \HH2.X.\R.$ and consider the associated Ricci-flat metric $\omega_a$ as in \cref{sec bochner}.
Let $0 \ne \vec v \in \HH0.X.\T X.$ be a nonzero holomorphic vector field.
Due to the Bochner principle, \cref{bochner}, $\vec v$ is parallel with respect to the Chern connection $D$ induced by $\omega_a$.
Dualizing using this metric, $\vec v$ gives rise to a parallel (hence holomorphic, as $D^{0,1} = \bar\del$) $1$-form $\alpha$ on $X$.
Clearly the contraction $\iota_{\vec v} \alpha \ne 0$.
The argument can also be read backwards, hence~\labelcref{contraction} is a perfect pairing and $\d\phi$ is an isomorphism.

Next, we show that the family $f_1 \from X \x_A A_1 \to A_1$ is trivial.
To this end, set $F \defn \alpha\inv(0)$.
Recalling that $A_1 = \Autn(X)$, we define a map $\wb\eta \from F \x A_1 \to X \x A_1$ by sending $(x, g) \mapsto \big( g(x), g \big)$.
Since $\alpha(g(x)) = \phi_g(\alpha(x)) = \phi_g(0) = \phi(g)$, the map $\wb\eta$ factors through the fibre product as $\eta \from F \x A_1 \to X \x_A A_1$.
The latter map has an inverse given by $\eta\inv(x, g) = \big( g\inv(x), g \big)$.
By construction, $\eta$ is a morphism over~$A_1$, \emph{i.e.}~it preserves the fibres of the projections to $A_1$.
It follows that $X \x_A A_1$ is isomorphic to $F \x A_1$ over $A_1$, as claimed.

What we have observed so far already implies that $X \to A$ is a surjective analytic fibre bundle, because the base change to $A_1$ has these properties.
It remains to see that the fibre $F$ is connected.
For this, consider the Stein factorization $X \xrightarrow{\;\beta\;} B \xrightarrow{\;\gamma\;} A$ of $\alpha$.
We claim that $\gamma$ is \'etale.
The question is local on $A$, so we may pass to a small open subset $U \subset A$ such that $\alpha\inv(U) \isom U \x F$ over $U$.
Then there is an isomorphism $\alpha_* \O{\alpha\inv(U)} \isom \O U^{\oplus N}$ in the category of finitely presented $\O U$-algebras, where $N$ is the number of connected components of $F$.
By construction of the Stein factorization, it follows that $\gamma\inv(U)$ is a disjoint union of $N$ copies of~$U$.
In particular, $\gamma$ is \'etale.

Consequently, $B$ is a complex torus too.
By the universal property of $A$, the map $X \xrightarrow{\;\beta\;} B$ factors via $A$ and we get a section $A \to B$ of $\gamma$.
This means that $\gamma$ is an isomorphism, \emph{i.e.}~$B = A$.
By the definition of Stein factorization, $X \to A = B$ has connected fibres.

The last statement, $q(X) \le \dim X$, follows from the surjectivity of $\alpha$ as follows:
\[ q(X) = \hh1.Y.\O Y. = \hh0.Y.\Omegap Y1. = \dim \Alb(Y) = \dim A \le \dim X, \]
where $Y \to X$ is a resolution of singularities.
\end{proof}

\begin{cor}[Torus covers] \label{torus cover}
Let $X$ be a normal compact \kahler space with klt singularities such that $\cc1X = 0$.
Then:
\begin{enumerate}
\item\label{tc.0} The augmented irregularity $\wt q(X) \le \dim X$.
In particular, it is finite.
\end{enumerate}
Furthermore, there exist normal compact \kahler spaces $T$ and $Z$ with canonical singularities together with a \qe cover $\gamma \from T \x Z \lto X$ such that:
\begin{enumerate}
\item\label{tc.1} $T$ is a complex torus of dimension $\wt q(X)$.
\item\label{tc.2} The canonical sheaf of $Z$ is trivial, $\can Z \isom \O Z$.
\item\label{tc.3} The augmented irregularity of $Z$ vanishes, $\wt q(Z) = 0$.
\end{enumerate}
Any cover $\gamma$ as above will be called a \emph{torus cover} of $X$.
\end{cor}

\begin{proof}
The proof closely follows along the lines of~\cite[Proposition~7.5]{GGK}, but with different references.
For later reference, and also as a courtesy to the reader, we give the argument here.

To begin with, note that by~\cite[Corollary~1.18]{JHM2}, the canonical divisor $K_X$ is torsion.
We may thus consider an index one cover $X_1 \to X$.
This has the additional property that $K_{X_1}$ is trivial (in particular, Cartier) and hence $X_1$ has canonical singularities.
That is, $X_1$ satisfies the assumptions of \cref{alb splits}.

For~\labelcref{tc.0}, note that $\wt q(X) = \wt q(X_1)$ by \cref{wt q invar}.
It is therefore sufficient to prove $\wt q(X_1) \le \dim X_1 = \dim X$.
But any \qe cover $X' \to X_1$ reproduces the assumptions of \cref{alb splits}, hence $q(X') \le \dim X' = \dim X_1$ by that result.
The claim follows by taking the supremum over all \qe covers of $X_1$.

Next, we construct a torus cover $\gamma$ as a sequence of quasi-\'etale covers as follows:
\[ T \x Z = X_3 \lto X_2 \lto X_1 \lto X. \]
The first map, $X_1 \to X$, is still the index one cover.
We have already seen that $\wt q(X_1)$ is finite.
Choose now a \qe cover $X_2 \to X_1$ that realizes this augmented irregularity, $q(X_2) = \wt q(X_1)$.
Finally, apply \cref{alb splits} to $X_2$ in order to obtain a further cover $X_3 \to X_2$ that splits off a torus as $X_3 = T \x Z$.
We need to show that~\labelcref{tc.1}--\labelcref{tc.3} hold.

By construction, $\dim(T) = q(X_3) = q(X_2) = \wt q(X_1) = \wt q(X)$, which proves~\labelcref{tc.1}.
For~\labelcref{tc.2}, since $\can{X_1}$ is already trivial, the same is true of $\can{T \x Z}$ and hence also of $\can Z$.
Finally, if there was a cover $Z' \to Z$ with $q(Z') > 0$, then $T \x Z' \to X$ would be a cover with irregularity $\dim T + q(Z') > \wt q(X)$, which contradicts the definition of $\wt q(X)$.
This shows~\labelcref{tc.3}.
\end{proof}

\begin{lem}[Invariance of $\wt q$] \label{wt q invar}
Let $Y \to X$ be a quasi-\'etale cover of normal compact complex spaces.
Then $\wt q(Y) = \wt q(X)$.
\end{lem}

\begin{proof}
Every \qe cover $Y' \to Y$ is, by composition, also a \qe cover of $X$.
The inequality $\wt q(Y) \le \wt q(X)$ is therefore obvious.
For the other direction, let $X' \to X$ be any \qe cover of $X$, and consider the normalized fibre product diagram
\[ \xymatrix{
Z \ar[rr] \ar[d] & & Y \ar[d] \\
X' \ar[rr] & & X.
} \]
Since all the maps in this diagram are \qe covers, we obtain
\[ \wt q(Y) \ge q(Z) \ge q(X'). \]
The claim follows by taking the supremum over all \qe $X' \to X$.
\end{proof}

\begin{prp}[Uniqueness of torus cover] \label{uniqueness torus cover}
In the setting of \cref{torus cover}, the spaces $T$ and $Z$ are unique up to \qe cover, in the following sense:
Suppose that $\gamma' \from T' \x Z' \to X$ is another torus cover of $X$.
Then:
\begin{enumerate}
\item\label{utc.1} The complex tori $T$ and $T'$ are isogeneous.
\item\label{utc.2} There is a common \qe cover $Z \leftarrow Z'' \to Z'$, where $Z''$ is likewise compact, connected and has canonical singularities.
\end{enumerate}
\end{prp}

\begin{proof}
Consider the normalized fibre product diagram
\[ \xymatrix{
Y \ar^{p'}[rr] \ar_-p[d] & & T' \x Z' \ar^-{\gamma'}[d] \\
T \x Z \ar^\gamma[rr] & & X.
} \]
Since $q(T \x Z)$ is already maximal, $q(Y) = q(T \x Z)$ and hence
\[ \alb(p) \from \Alb(Y) \lto \Alb(T \x Z) = T \]
is finite, \emph{i.e.}~an isogeny.
The same argument applies to $T' \x Z'$.
So both $T$ and $T'$ are isogeneous to $\Alb(Y)$.
This implies~\labelcref{utc.1} because isogeny is an equivalence relation.

We turn to~\labelcref{utc.2}.
Consider an arbitrary $s \in \Alb(Y)$.
Set $t \defn \alb(p)(s) \in T$ and $Z_t \defn \set t \x Z \subset T \x Z$.
We claim that we can take $Z'' \defn \alb_Y\inv(s)$.
By \cref{alb splits}, $Z''$ is connected and has canonical singularities.
It is clear that $p(Z'') \subset Z_t$.
Then they are equal, because they have the same dimension and $Z_t$ is irreducible.
Hence the restriction of $p$ is a \qe cover $Z'' \to Z_t \isom Z$.
The same argument also shows that $p'$ restricts to a \qe cover $Z'' \to Z'$.
This proves the claim.
\end{proof}

The Abundance conjecture for compact \kahler spaces with canonical (or klt) singularities and vanishing first Chern class has been established recently (\emph{cf.} \emph{e.g.} \cite[Corollary~1.18]{JHM2} and references therein) building upon delicate results on jump loci for cohomology. Below, we explain how to recover the result based solely on the Albanese splitting, \emph{i.e.} \cref{alb splits}. 

\begin{cor}[Special case of Abundance] \label{abundance c1=0}
Let $X$ be a normal compact \kahler space with \emph{canonical} singularities and $\cc1X = 0$.
Then $\can X$ is torsion, \emph{i.e.}~$\can X^{[m]} \isom \O X$ for some $m > 0$.
\end{cor}

\begin{proof}
It follows from \cref{alb splits} that $\wt q(X) \le \dim X$ is finite (\emph{cf.}~the proof of \cref{torus cover}).
Choose a \qe cover $X_1 \to X$ such that $q(X_1) = \wt q(X)$, and apply \cref{alb splits} to $X_1$.
This yields an \'etale cover $X_2 \to X_1$ that splits as $X_2 = T \x Z$, where $T$ is a complex torus of dimension $q(X_1)$.
As $q(T) + q(Z) = q(X_2) = q(X_1) = q(T)$, we see that $q(Z) = 0$.

Now note that $\hh1.Z.\O Z. \le q(Z)$ by the Leray spectral sequence.
(In fact, equality holds because $Z$ has rational singularities.)
Thus $q(Z) = 0$ implies, via the exponential sequence, that any line bundle on $Z$ with vanishing integral first Chern class is trivial.
Consequently, for any \Q-Cartier reflexive rank one sheaf $\sA$ on $Z$ with vanishing \emph{real} first Chern class, there is $m > 0$ such that $\sA^{[m]} \isom \O Z$.
We may apply this observation to $\sA = \can Z$ and combine it with the fact that $\can T$ is trivial.
The conclusion is that $\can{X_2} = \can{T \x Z}$ is torsion as well.
Then the same is true of $\can X$.
\end{proof}

\begin{cor}[Characterization of torus quotients via $\wt q$] \label{tq irreg}
Let $X$ be a normal compact \kahler space with klt singularities such that $\cc1X = 0$.
Then the following conditions are equivalent:
\begin{enumerate}
\item\label{tqi.1} The augmented irregularity attains its maximum possible value, namely $\wt q(X) = \dim X$.
\item\label{tqi.2} There exists a complex torus $T$ and a holomorphic action of a finite group $G \acts T$, free in codimension one, such that $X \isom \factor TG$.
\end{enumerate}
\end{cor}

\begin{proof}
``\labelcref{tqi.1} $\imp$ \labelcref{tqi.2}'': Consider a torus cover $\gamma \from T \x Z \to X$.
The dimension of $Z$ equals $\dim X - \dim T = \dim X - \wt q(X) = 0$, so $Z$ is a point.
This means that $\gamma$ is a finite \qe cover of $X$ by the complex torus $T$.
The claim follows by taking Galois closure, \emph{cf.}~\cref{lem:galois closure} and~\cite[Lemma~7.4]{TorusQuotients}.

``\labelcref{tqi.2} $\imp$ \labelcref{tqi.1}'': Obvious, as $\wt q(X) \ge q(T) = \dim T = \dim X$.
\end{proof}

\section{Spaces with flat tangent sheaf on the smooth locus} \label{sec TX flat}

The aim of this section is to prove the following result, which comes in a local as well as a global version.
This is an important step towards \cref{intro:torusquotient}.

\begin{thm}[Compact spaces with flat tangent sheaf] \label{global flat tangent}
Let $X$ be a normal compact complex space with klt singularities.
Assume that the tangent sheaf of the smooth locus $\T{\Reg X}$ is flat in the sense of \cref{flat sheaf}.
Then $X$ admits a \qe Galois cover $\wt X \to X$ with $\wt X$ smooth.
\end{thm}

\begin{thm}[Germs with flat tangent sheaf] \label{local flat tangent}
Let $(X, x)$ be a germ of a klt singularity such that $\T{\Reg X}$ is flat.
Then $(X, x)$ is a quotient singularity.
\end{thm}

The notion of \emph{\gkp covers}, defined next, is central to the proof.

\begin{dfn} \label{dfn gkp}
Let $X$ be a normal complex space.
A \emph{\gkp cover} of $X$ is a quasi-\'etale Galois cover $\gamma \from \wt X \to X$ satisfying the following equivalent conditions:
\begin{enumerate}
\item\label{gkp.1} Any \'etale cover of $\Reg{\wt X}$ extends to an \'etale cover of $\wt X$.
\item\label{gkp.2} Any quasi-\'etale cover of $\wt X$ is \'etale.
\item\label{gkp.3} The natural map of \'etale fundamental groups $\piet{\Reg{\wt X}} \to \piet{\wt X}$ induced by the inclusion $\Reg{\wt X} \inj \wt X$ is an isomorphism.
\end{enumerate}
\end{dfn}

\noindent
The equivalence of conditions~\labelcref{gkp.1}--\labelcref{gkp.3} is discussed in~\cite[Theorem~1.5]{GKP16} and its proof. 

\begin{rem} \label{rem:extension}
If $\wt X \to X$ is a \gkp cover, then it follows from Malcev's theorem that any linear representation of $\pi_1(\wt X_{\rm reg})$ factors through $\pi_1(\wt X)$, or equivalently, any flat bundle over $\wt X_{\rm reg}$ extends to a flat bundle on $\wt X$, \emph{cf.}~\cite[\textsection~8.1]{GKP16}.
\end{rem}

In general, a normal space $X$ will not admit a \gkp cover (for an easy example, take a cone over an elliptic curve).
It was shown in~\cite[Theorem~1.1]{GKP16} that such a cover exists if $X$ is an \emph{algebraic} variety with klt singularities.
Also, it is relatively easy to see that existence of the cover would imply \cref{global flat tangent}, \emph{cf.}~\cref{gkp smooth} below.
Our \cref{gkp max strata} is a weaker version of~\cite{GKP16} for complex spaces: we construct a \gkp cover ``generically'', \emph{i.e.}~only over general points of the singular locus $\Sing X$.
This, however, is sufficient to prove the main result.

\subsection{Technical preparations}

As mentioned above, \cref{global flat tangent} is easy to prove if a \gkp cover is already known to exist.

\begin{prp} \label{gkp smooth}
Let $X$ be a normal complex space with log canonical singularities, and assume that the tangent sheaf of the smooth locus $\T{\Reg X}$ is flat.
Then any \gkp cover of $X$ is smooth (if it exists).
\end{prp}

\begin{proof}
Let $\gamma \from \wt X \to X$ be a \gkp cover.
The pullback $\gamma^* \T{\Reg X}$ is a flat sheaf on $\gamma\inv(\Reg X)$, which satisfies $\gamma\inv(\Reg X) \subset \Reg{\wt X}$ and the complement has codimension at least two.
Hence $\gamma^* \T{\Reg X}$ extends to a flat sheaf on $\Reg{\wt X}$
and we can extend this sheaf further, to a flat sheaf $\sF$ on all of $\wt X$, \emph{cf.} \cref{rem:extension}.
By reflexivity, $\sF \isom \T{\wt X}$ and in particular $\T{\wt X}$ is locally free.
The solution of the Lipman--Zariski conjecture for log canonical spaces (\cite[Corollary~1.3]{GK13}, \cite[Theorem~1.1]{DruelZL}) then shows that $\wt X$ is smooth.
\end{proof}

The following lemma describes the \'etale fundamental group of the link of maximal dimensional strata of a klt space. 

\begin{lem} \label{lem:finite pi_1 max stratum}
Let $X$ be a klt complex space and let us choose a Whitney stratification $\sS$ of $X$.
The maximum dimensional strata are denoted by $X_\lambda$, $\lambda \in \Lambda_\maxi$, as in \cref{def:max strata}.
For any $\lambda \in \Lambda_\maxi$, the group $\piet{\Link X.X_\lambda.}$ is finite.
\end{lem}

\begin{proof}
Pick a normal slice $N$ of $X_\lambda$ cutting the stratum in the unique point $p$.
Set $N_X \defn N \cap X$ and $N^* \defn N_X \setminus \set p$.
From \cref{thm:neighborhood stratum} we know that $N^*$ has the homo\-topy type of $\Link X.X_\lambda.$ and in particular $\piet{\Link X.X_\lambda.} = \piet{N^*}$.
But now according to~\cite[Lemma~5.12]{KM98}, $(N_X, p)$ is an isolated klt singularity, hence algebraic \labelcref{example isolated singularities}.
We can now appeal to~\cite[Theorem~1]{Xu14} and conclude that the algebraic local fundamental group $\piet{N_X \setminus \set p}$ is finite.
\end{proof}

\begin{rem}\label{rem:link max strata}
The lemma above should apply to any stratum, not only to the maximal dimensional ones.
Unfortunately, Xu's result is only available in the algebraic setting, see also \cref{rem:existence GKP cover}.
This is why we stick to maximal dimensional strata: they are the ones whose slices have only isolated singularities.
\end{rem}

We now observe that quasi-\'etale covers are well understood over the maximal dimensional strata.

\begin{lem} \label{lem:Galois cover group link}
Let $\gamma \from Y \to X$ be a \qe Galois cover between klt complex spaces, $X$ being endowed with a Whitney stratification $\sS$.
Fix an index $\lambda \in \Lambda_\maxi$ with respect to $\sS$.
Then:
\begin{enumerate}
\item\label{gcgl.1} The cover $\gamma$ naturally induces a subgroup $G_\lambda(\gamma) \subset \piet{\mathrm{Link}(X, X_\lambda)}$, well-defined up to conjugation.
\item\label{gcgl.2} If $g \from Z \to Y$ is any further \qe cover such that the composition $\gamma \circ g \from Z \to X$ is also Galois, then $G_\lambda(\gamma \circ g) \subset G_\lambda(\gamma)$.
\end{enumerate}
Furthermore, there is a dense Zariski open subset $U_\lambda \subset X_\lambda$, which depends on $\gamma$ but not on $g$, such that $X_\lambda \setminus U_\lambda$ is analytic and the following holds:
\begin{enumerate}
\item\label{gcgl.3} We have equality $G_\lambda(\gamma \circ g) = G_\lambda(\gamma)$ if and only if $g$ is \'etale over $\gamma\inv(U_\lambda)$.
\end{enumerate}
\end{lem}

\begin{proof}
By \cref{strat map}, we can choose Whitney stratifications $\sS' < \sS$ and $\sS_Y$ of $X$ and $Y$, respectively, such that $\gamma \from (Y, \sS_Y) \to (X, \sS')$ is a stratified map.
We denote by $X'_\lambda \subset X_\lambda$ the unique stratum of $\sS'$ contained in $X_\lambda$ as an open subset.

The map $\gamma$ being stratified, $\gamma\inv(X'_\lambda)$ is a union of strata of $\sS_Y$ that are of maximum dimension.
Let us fix a normal slice $N$ of $X'_\lambda$ such that $\gamma\inv(N)$ is still a union of normal slices.
We pick a component $T$ of $\gamma\inv(X'_\lambda)$ and denote by $N_T$ the corresponding slice.
The map induced on the punctured slices (cut with $X$ and~$Y$, respectively) $N_T^* \to N^*$ is finite \'etale and thus corresponds to a finite index subgroup
\[ G_\lambda(\gamma) \subset \piet{N^*} = \piet{\mathrm{Link}(X, X'_\lambda)} = \piet{\mathrm{Link}(X, X_\lambda)}. \]
The morphism $\gamma$ being Galois, this subgroup is independent of the choice of the component $T$.
It does however depend on the choice of a basepoint in $N_T^*$, which as usual we have suppressed.
This is why $G_\lambda(\gamma)$ is only well-defined up to conjugation.
This proves the first assertion, and the second one is then clear by construction.

Concerning the third assertion, we claim that we can take $U_\lambda = X'_\lambda$.
Indeed, it is clear that the groups in question are equal if $g$ is \'etale over $\gamma\inv(U_\lambda)$.
For the other direction, pick stratifications $\sS_Z$ of $Z$ and $\sS_Y'$ of $Y$ such that $\sS_Y'$ is a refinement of $\sS_Y$ and $g$ becomes stratified, as above.
Write $\gamma\inv(U_\lambda)$ as a union of $\sS_Y$-strata,
\[ \gamma\inv(U_\lambda) = \bigcup_{\mu \in M} Y_\mu, \]
and let $Y_\mu'$ be the unique $\sS_Y'$-stratum that is open in $Y_\mu$.
It follows from Corollary~3.16 in the preprint version of~\cite{GKP16} that $g$ is \'etale over $Y_\mu'$, for each $\mu \in M$.
But then $g$ is even \'etale over all of $Y_\mu$, by \cref{cor:branching stratum} applied to the original stratification $\sS_Y$.
In other words, $g$ is \'etale over $\gamma\inv(U_\lambda)$, as desired.
\end{proof}

With the previous results at hand, we can prove the existence of \gkp covers, at least after discarding a sufficiently small analytic subset.

\begin{prp} \label{gkp max strata}
Let $X$ be a klt complex space.
Assume that $X$ is either compact or a germ.
Then there exists an analytic subset $Z \subset \Sing X$ with $\dim Z \le \dim \Sing X - 1$ such that $X^\circ \defn X \setminus Z$ admits a \gkp cover.
\end{prp}

If $X$ is compact with only \emph{isolated} klt singularities, then we have $\dim Z = -1$ in the above statement, which means $Z = \emptyset$.
Hence $X$ itself admits a \gkp cover.

%
%

\begin{proof}[Proof of \cref{gkp max strata}]
Fix a Whitney stratification $\sS$ of $X$ such that $S_d = \Sing X$.
If $X$ is compact, then $\Lambda$ and in particular $\Lambda_\maxi$ are finite sets.
If $X$ is a germ, we may assume $\Lambda$ to be finite after shrinking $X$, by the local finiteness of $\sS$.
First we show the following.

\begin{clm} \label{1083}
There exists a \qe Galois cover $\gamma \from Y \to X$ such that for every further cover $g \from Z \to X$ with $\gamma \circ g$ Galois and any $\lambda \in \Lambda_\maxi$, we have $G_\lambda(\gamma \circ g) = G_\lambda(\gamma)$.
\end{clm}

\begin{proof}
Assuming that the claim is false, we can construct an infinite tower of \qe covers
\[ X = Y_0 \xleftarrow{\;\gamma_1\;} Y_1 \xleftarrow{\;\gamma_2\;} Y_2 \xleftarrow{\;\gamma_3\;} \cdots \]
such that for each $i \ge 1$, the map $\eta_i = \gamma_1 \circ \cdots \circ \gamma_i \from Y_i \to X$ is Galois and for some index $\lambda(i) \in \Lambda_\maxi$, the inclusion $G_{\lambda(i)}(\eta_{i+1}) \subset G_{\lambda(i)}(\eta_i)$ is strict.
Because $\Lambda_\maxi$ is finite, some $\lambda_0 \in \Lambda_\maxi$ has to appear as $\lambda(i)$ for infinitely many values of~$i$.
This yields a contradiction to the finiteness of $\piet{\Link X.X_{\lambda_0}.}$ given by \cref{lem:finite pi_1 max stratum}.
\end{proof}

Applying~\labelcref{gcgl.3} to the cover $\gamma$ given by \cref{1083} and to all $\lambda \in \Lambda_\maxi$, we obtain dense open subsets $U_\lambda \subset X_\lambda$ with the following property:
any \qe cover $W \to Y$ such that $W \to X$ is Galois is \'etale over each $\gamma\inv(U_\lambda)$.
If $W \to X$ is not Galois, the conclusion still holds because we may replace $W$ by its Galois closure, \emph{cf.}~\cref{lem:galois closure}.
We now consider $X^\circ = X \setminus Z$, with
\[ Z \defn S_{d-1} \; \cup \bigcup_{\lambda \in \Lambda_\maxi} (X_\lambda \setminus U_\lambda), \]
and claim that $Y^\circ \defn \gamma\inv(X^\circ) \to X^\circ$ is a \gkp cover.
Noting that $Z \subset \Sing X$ has dimension less than $\dim \Sing X$, this will finish the proof.

We aim to verify condition~\labelcref{gkp.2} for $Y^\circ$.
That is, we need to show that any \qe cover $g^\circ \from W^\circ \to Y^\circ$ is in fact \'etale.
Given such a cover, we may extend it to a \qe cover $g \from W \to Y$.
As we have seen above, $g$ is \'etale over each $\gamma\inv(U_\lambda)$.
But we can write $Y$ as a disjoint union
\[ Y = \gamma\inv(Z) \cup \bigcup_{\lambda \in \Lambda_\maxi} \gamma\inv(U_\lambda) \cup \gamma\inv(\Reg X) \]
and hence also $Y^\circ = \bigcup_\lambda \gamma\inv(U_\lambda) \cup \gamma\inv(\Reg X)$.
Recalling from purity of branch locus that $\gamma\inv(\Reg X) \subset \Reg Y$ and that $g$ is \'etale over $\Reg Y$, we conclude that $g$ is \'etale over all of $Y^\circ$.
In other words, $g^\circ$ is \'etale.
This ends the proof.
\end{proof}

%

\subsection{Proof of \cref{global flat tangent}}

We proceed by induction on the dimension of the singular locus $\Sing X$.
If $\dim \Sing X = -1$, \emph{i.e.}~$\Sing X$ is empty, then we may simply take $\wt X = X$.
Otherwise we apply \cref{gkp max strata} to obtain an open subset $\Reg X \subset X^\circ \subset X$ which has a \gkp cover $\gamma^\circ \from \wt{X^\circ} \to X^\circ$ and satisfies $\dim(X \setminus X^\circ) \le \dim \Sing X - 1$.
We may restrict $\gamma^\circ$ to an \'etale cover $\gamma'$ of $\Reg X$, and then in turn extend $\gamma'$ to a \qe cover of $X$, say $\gamma \from \wt X \to X$.
By the uniqueness part of~\cite[Theorem~3.4]{DG94}, the map $\gamma$ will be an extension of $\gamma^\circ$.
\[ \xymatrix{
X' \ar@{ ir->}[rr] \ar_-{\gamma'}[d] & & \wt{X^\circ} \ar@{ ir->}[rr] \ar_-{\gamma^\circ}[d] & & \wt X \ar^-\gamma[d] \\
\Reg X \ar@{ ir->}[rr] & & X^\circ \ar@{ ir->}[rr] & & X
} \]
The space $\wt{X^\circ}$ is smooth thanks to \cref{gkp smooth}.
Thus we see that $\dim \Sing{\wt X} \le \dim \Sing X - 1$.
Since $\T{\Reg{\wt X}}$ is still flat, the induction hypothesis applied to $\wt X$ yields a \qe cover $\wh X \to \wt X$ with $\wh X$ smooth.
We arrive at the desired smooth cover of $X$ by taking the Galois closure of the composed map $\wh X \to X$. \qed

\subsection{Proof of \cref{local flat tangent}}

We still argue by induction on $\dim \!\big( \Sing{(X,x)} \big)$ as above.
If $(X, x)$ is smooth, there is nothing to show.
We just have to check that we can apply the induction hypothesis: to start with, we choose $X$ as a small neighborhood of $x$ with a finite Whitney stratification.
We can then find a \qe cover $f \from \wt X \to X$ with $\dim(\Sing{\wt X}) < \dim(\Sing X)$.
It is enough to pick a point $\wt x \in f\inv(x) \subset \wt X$: this $\wt x$ has a neighborhood $\wt U$ that is a quotient singularity.
The open neighborhood $U \defn f(\wt U)$ of $x$ has a smooth \qe cover and $(X, x)$ is then a quotient singularity according to \cref{lem:quotient sing}. \qed

\section{Holonomy of singular Ricci-flat metrics} \label{sec:holonomy} \label{section holonomy}

The goal of this section is to prove \cref{intro:hol}, which combined with the Bochner principle leads to \cref{intro:structure}.

\subsection{A vanishing lemma}

The following lemma will be used repeatedly.
The upshot is that a vanishing assumption on $X$ implies a certain vanishing result on~$\Reg X$.

\begin{lem}[Vanishing lemma] \label{lem:virab}
Let $X$ be a compact \kahler space with klt singularities such that $q(X) = 0$.
Then one has
\[ \HH1.\Reg X.\C. = 0. \]
If $X$ additionally satisfies $\wt q(X) = 0$, then any representation $\rho \from \pi_1(\Reg X) \to G$ with virtually abelian image actually has finite image.
\end{lem}

\begin{exm} \label{kummer}
The example of the (simply connected) singular Kummer surface $X = \factor T{\pm1}$ where $T$ is a complex $2$-torus shows that one can have $q(X) = 0$ while $\pi_1(\Reg X)$ is virtually abelian, yet infinite.
Indeed, there is a short exact sequence
\[ 0 \lto \Z^4 \lto \pi_1(\Reg X) \lto \factor{\Z}{2 \Z} \lto 0. \]
Of course, in this example $\wt q(X) = 2 > 0$.
\end{exm}

\begin{rem}
The proof of the Lemma combined with~\cite{KS} shows more generally that if $X$ is compact \kahler with klt singularities and $\wh X \to X$ is a resolution, then one has a (non-canonical) isomorphism
\[ \HH1.\Reg X.\C. = \HH1.\wh X \setminus E.\C. \isom \HH0.X.\Omegar X1. \oplus \HH1.X.\O X.. \]
On the other hand, since $X$ has only rational singularities, one also has
\[ \HH1.X.\C. \isom \HH1.\wh X.\C. = \HH0.\wh X.\Omegap{\wh X}1. \oplus \HH1.\wh X.\O{\wh X}. \]
canonically.
Hence the restriction map $\HH1.X.\C. \to \HH1.\Reg X.\C.$ is an isomorphism.
Equivalently, the \emph{abelianization} of the natural surjection $\pi_1(\Reg X) \to \pi_1(X)$ has finite kernel.
Cf.~\cref{kummer} above.
\end{rem}

\begin{proof}[Proof of \cref{lem:virab}]
Let $\wh X \to X$ be a projective log resolution of $X$, isomorphic over $X_{\rm reg}$, and let $E = \sum_i E_i$ be the exceptional divisor.
In what follows, one will identify $\Reg X$ with $\wh X \setminus E$.
One has
\begin{equation} \label{H11}
\HH1.\wh X.\O{\wh X}. = 0
\end{equation}
since $q(X) = 0$, \emph{cf.}~\cref{dfn irreg}.
Since $\wh X$ is a compact \kahler manifold, one also has $\HH0.\wh X.\Omegap{\wh X}1. = 0$. By \cite{KS}, this implies that $\HH0. X.\Omegap{ X}{[1]}. =\HH0.\wh X\setminus E.\Omegap{\wh X}1. =  0$ and, in particular, one gets
\begin{equation} \label{logvanishing}
\HH0.\wh X.\Omegap{\wh X}1(\log E). = 0.
\end{equation}
By Deligne's theorem~\cite[Theorem~8.35(b)]{Voi02}, there is a short exact sequence
\[ 0 \lto \HH0.\wh X.\Omegap{\wh X}1(\log E). \lto \HH1.\wh X \setminus E.\C. \lto \HH1.\wh X.\O{\wh X}. \lto 0, \]
so that~\labelcref{H11} and~\labelcref{logvanishing} yield the requested vanishing
\begin{equation} \label{xregvanishing}
\HH1.\Reg X.\C. = \HH1.\wh X \setminus E.\C. = 0.
\end{equation}
Using the universal coefficient theorem, this implies that the finitely generated abelian group $\Hh1.\Reg X.\Z.$ has rank zero, hence is finite.

Let us now move on to the second part of the lemma.
Without loss of generality, one can assume that $\rho$ is surjective, \emph{i.e.}~$\rho \big( \pi_1(\Reg X) \big) = G$.
Under the assumptions, there exists a finite index normal subgroup $H \trianglelefteq \pi_1(\Reg X)$ such that $\rho(H) \subset G$ is abelian.
The group $H$ is realized as the fundamental group of a \kahler manifold $Y^\circ$ equipped with an \'etale cover $f^\circ \from Y^\circ \to \Reg X$, \emph{i.e.}~$f^\circ_* \big( \pi_1(Y^\circ) \big) = H$.
One can extend $f^\circ$ to a quasi-\'etale cover $f \from Y \to X$ with $Y$ a compact \kahler space with klt singularities.
The inclusion $Y^\circ \subset \Reg Y$ induces an isomorphism $\pi_1(Y^\circ) \isom \pi_1(\Reg Y)$ since $\Reg Y \setminus Y^\circ \subset f\inv(\Sing X)$ has codimension at least two in $\Reg Y$.

Now, one has $q(Y) = \wt q(X) = 0$, hence $\HH1.\Reg Y.\Z.$ is finite by the first part of the lemma.
Since the image of the representation $\tau \defn \rho \circ f^\circ_*$ of $\pi_1(Y^\circ) \isom \pi_1(\Reg Y)$ is abelian, it factors through $\HH1.\Reg Y.\Z.$ and is therefore finite.
The following diagram visualizes the argument.
\[ \xymatrix{
\pi_1(Y^\circ) \ar_-{f^\circ_*}[d] \ar^-\tau[drr] & & \\
\pi_1(\Reg X) \ar_-\rho[rr] & & G
} \]
The lemma follows since $f^\circ_* \big( \pi_1(Y^\circ) \big) \trianglelefteq \pi_1(\Reg X)$ has finite index.
\end{proof}

\subsection{Holonomy of the non-flat factors}

We fix the following setup.

\begin{setup} \label{setup2}
Let $X$ be a $n$-dimensional complex, compact \kahler space with canonical singularities and trivial canonical sheaf, $K_X \sim 0$.
Next, we fix a \kahler class $\alpha \in \HH2.X.\R.$ and consider the unique singular Ricci-flat metric $\omega_\alpha \in \alpha$ as in \cref{setup1}.
\end{setup}

It follows from~\cite[Theorem~A]{GSS} (see also~\cite[Theorem~8.1]{GGK}, where the assumption on the projectivity of $X$ is not used, \emph{cf.}~\cref{comparison}) that one can decompose the tangent sheaf as
\begin{equation} \label{decomp}
\T X = \bigoplus_{i \in I} \sE_i
\end{equation}
where the $\sE_i$ are stable bundles with slope zero (with respect to $\alpha$) and are such that $\sE_i\big|_{\Reg X}$ is a parallel subbundle of $\sT_{\Reg X}$ with respect to $\omega_\alpha$.
We set $n_i \defn \rk(\sE_i)$.
The holonomy group $G \defn \mathrm{Hol} \big( \T{\Reg X}, \omega_\alpha \big)$ can be decomposed as a product $G = \prod_i G_i$ where $G_i = \mathrm{Hol} \big( \sE_i\big|_{\Reg X}, \omega_\alpha \big)$ is irreducible.

By~\cite[Proposition~7.3]{GGK}, there exists a \qe cover $f \from Y \to X$ such that the decomposition induced by the standard representation of the holonomy group $G_Y \defn \mathrm{Hol} \big( \sT_{\Reg Y}, f^*\omega_{\alpha} \big)$ is a refinement of the one of the identity component~$G_Y^\circ$.
Note that the construction of this so-called \emph{weak holonomy cover} above is independent of the projectivity assumption in~\cite{GGK}.
From now on, we will replace $X$ by $Y$ and work with the following

\begin{awlog}
The restricted holonomy $G_i^\circ$ of each factor $\sE_i$ in~\labelcref{decomp} is either trivial or irreducible.
\end{awlog}

By~\cite[Proposition~5.3]{GGK}, there are only three possibilities for $G_i^\circ$: one has
\[ G_i^\circ =
\begin{cases}
\set 1 & \text{or} \\
\mathrm{SU}(n_i), \qquad n_i \ge 3, & \text{or} \\
\mathrm{Sp}(n_i/2), \quad\, \text{$n_i \ge 2$ even}.
\end{cases} \]
We set $J \defn \set{i \in I \;\big|\; G_i^\circ = \set 1}$, $K \defn I \setminus J$ and $\sF \defn \bigoplus_{j\in J} \sE_j$.
In other words, one has
\begin{equation} \label{decompt}
\T X = \sF \oplus \bigoplus_{k \in K} \sE_k
\end{equation}
with $\mathrm{Hol}^\circ \big( \sF\big|_{\Reg X}, \omega_\alpha \big) = \set 1$ while the full holonomy group $G_k = \mathrm{Hol} \big( \sE_k\big|_{\Reg X}, \omega_\alpha \big)$ satisfies $G_k^\circ = \mathrm{SU}(n_k)$ or $G_k^\circ = \mathrm{Sp}(n_k/2)$.
In either case, one has
\begin{equation} \label{u1}
\factor{G_k}{G_k^\circ} \subset \mathrm{U}(1)
\end{equation}
by the proof of~\cite[Lemma~7.8]{GGK}.
Consider a torus cover $\gamma \from T \x Z \to X$; then $\mathrm{pr}_T^* \T T$ is a direct summand of $\gamma^{[*]} \sF$, so that $\gamma^{[*]} \sE_k$ is canonically identified with a direct summand of $\mathrm{pr}_Z^* \sT_Z$.
Hence one can realize $G_k$ as an irreducible factor of the holonomy group of $\Reg Z$.
Since $\wt q(Z) = 0$, the canonical surjection
\begin{equation*}
\rho \from \pi_1(\Reg Z) \twoheadrightarrow \factor{G_k}{G_k^\circ}
\end{equation*}
combined with \cref{lem:virab} leads to the following

\begin{prp} \label{hol gk}
For any $k \in K$, the group $\factor{G_k}{G_k^\circ}$ is finite. \qed
\end{prp}

\begin{cor} \label{corhol}
In \cref{setup2} and up to performing a \qe cover, the \emph{full} holonomy group of a non-flat, irreducible factor $\sE_k$ in the decomposition~\labelcref{decompt} is either $\mathrm{SU}(n_k)$ or $\mathrm{Sp}(n_k/2)$. \qed
\end{cor}

We can now move on to the

\begin{proof}[Proof of \cref{intro:hol}] \label{proof theorem holonomy}
Starting from \cref{setup2}, we pass to an additional cover of $X$ so that the conclusion of \cref{corhol} holds. \\

\textit{First item.} Recall that the sheaves $\sE_i$ from~\labelcref{decomp} are such that $\sE_i\big|_{\Reg X}$ are parallel with respect to $\om$. 
Since the Chern connection $D$ is torsion-free, we have $[u, v] = D_u v - D_v u$ for any local smooth vector fields $u,v$ on $\Reg X$. Applying this formula to local sections of $\sE_i\big|_{\Reg X}$ and keeping in mind that the latter subbundle is parallel, we see that $\sE_i\big|_{\Reg X}$ is preserved under Lie bracket. Moreover, $\sE_i$ is saturated in $\sT_X$, being a direct summand of $\sT_X$.
Combining those two facts, we get that the image of the Lie bracket $[-,-] \from \bigwedge^2 \sE_i \to \factor{\T X}{\sE_i}$ is a torsion-free sheaf supported on $\Sing X$, hence it vanishes.
This shows that $\sE_i$ defines indeed a foliation on $X$.
Although a direct sum of foliations needs not be integrable in general, the argument above still applies and shows that $\sF = \bigoplus_{j\in J} \sE_j$ is a foliation as well.
Next, we know that $K_X$ is a trivial line bundle and that for any $k \in K$, one has $\det \sE_k \isom \O X$ since the holonomy of $\sE_k$ lies in $\mathrm{SU}(n_k)$ by \cref{corhol}.
This implies that $\det \sF \isom \O X$.

Note that to prove that the sheaves $\sE_i$ are foliations, we could have used the last result (\emph{i.e.} $\det  \sE_i \isom \O X$) and reproduced the arguments of \cite[Theorem~7.11]{GKP} instead of appealing to the Bochner principle. \\

\textit{Second item.} This is clear from the definition of $\sF$ and the first item. \\

\textit{Third item.} The statement on the holonomy groups is contained in \cref{corhol}. Now, we claim that $\sE_k$ is strongly stable with respect to $\alpha$. Since restricted holonomy does not change after passing to a \qe cover, it is enough to show that $\sE_k$ is stable with respect to $\alpha$. But if it were not the case, any saturated destabilizing subsheaf would be parallel with respect to $\om_\alpha$ on $\Reg X$ by Bochner's principle for bundles, \emph{cf.} \cref{comparison}. Therefore, the group $\mathrm{Hol}(\sE_k|_{\Reg X},\om_\alpha)$ would not act irreducibly on $\mathbb C^{n_k}$, which is a contradiction.  
Note that if one only wants to show strong stability of $\sE_k$ (and not compute its holonomy), then the argument becomes quite simpler and relies only on the polystability of $\sT_X$, \emph{cf.} \cite[Corollary~7.3]{GKP}.

Now, let $\beta \in \HH2. X. \mathbb R.$ be a \kahler class and let us show that $\sE_k$ is strongly stable with respect to $\beta$. As before, it is enough to show stability. Let $\om_\beta\in \beta$ be the singular Ricci-flat metric. 
Since $c_1(\sE_k)=0$, a (saturated) $\beta$-destabilizing subsheaf $0\neq \sG\subsetneq \sE_k$ satisfies $c_1(\sG)\cdot \beta^{n-1}=0$. By polystability of $\sT_X$ with respect to $\beta$, we have a holomorphic splitting $\sT_X=\sG\oplus \sG^\perp$ over $\Reg X$ which extends over $X$ by reflexivity, \emph{cf.} proof of \cite[Theorem~A]{GSS} or \cite[Claim~9.16]{GGK}. Here, $\perp$ is meant with respect to $\om_\beta$. Since $\sT_X$ is semistable with respect to $\alpha$, we have necessarily $c_1(\sG)\cdot \alpha^{n-1}=0$. This contradicts the $\alpha$-stability of $\sE_k$ established earlier. \\

\textit{Fourth item.} Assume that $X = T \x Z$, where $T$ is a complex torus and $\wt q(Z) = 0$.
We need to show that the decomposition $\T X = \mathrm{pr}_T^*\sT_T \oplus \mathrm{pr}_Z^* \T Z$ is such that $\mathrm{pr}_T^* \T T$ is a direct summand of $\sF$.
The Bochner principle for bundles (\emph{cf.}~\cref{comparison}) shows that over $\Reg X = T \x \Reg Z$, both summands are parallel subbundles and, in particular, $\om$ can be decomposed as $\om_\alpha = \mathrm{pr}_T^* \, \om_T \oplus \mathrm{pr}_Z^* \, \om_Z$ where $\om_T$ (resp.~$\om_Z$) is a \kahler Ricci-flat metric on $T$ (resp.~$\Reg Z$).
Since a Ricci-flat \kahler metric on a torus is necessarily flat, we have $\mathrm{Hol}^\circ \big( \mathrm{pr}_T^* \T T \big|_{\Reg X}, \om_{\alpha} \big) = \mathrm{Hol}^\circ(T, \om_T) = \set 1$.
The first part of the statement now follows easily.
The second part follows from \cref{prop:holonomy flat factor} below.
\end{proof}

\subsection{On the flat factor}

Since $\T T$ is a direct summand of $\sF$, it is clear that $\wt q(X) \le \rk \sF$.
We conjecture that equality always holds.
In particular, if $\wt q(X) = 0$ then the flat factor $\sF$ in the decomposition~\labelcref{decompt} should be zero.
In the projective case, the conjecture was established by~\cite[Corollary~7.2]{GGK} as a consequence of Druel's algebraic integrability result for flat sheaves~\cite{Dru16}.
We are able to prove two partial results in this direction:

\begin{enumerate}
\item If $\rk \sF = n \defn \dim X$, then also $\wt q(X) = n$.
This is a direct consequence of \cref{flat tangent implies torus} below, since $\rk \sF = n$ means that $\T{\Reg X}$ is flat and torus quotients obviously have $\wt q(X) = n$.
\item The conjecture can be derived from the complex space version of~\cite{GKP16}, \emph{cf.}~\cref{prop:holonomy flat factor} and \cref{rem:existence GKP cover}.
\end{enumerate}

\begin{prp} \label{flat tangent implies torus}
Let $X$ be a normal compact \kahler space that has only klt singularities.
If $\T{\Reg X}$ is flat, then $X$ is a torus quotient.
\end{prp}

\begin{proof}
According to \cref{global flat tangent}, $X$ has a smooth \qe cover $\wt X$.
This compact space $\wt X$ is still \kahler~\cite[Corollary~3.2.2]{Var} and, by construction, the tangent bundle of $\wt X$ is flat.
We can then apply the classical characterization of torus quotients derived from Yau's theorem~\cite{Yau78} and the uniformization theorem, \emph{cf.}~\emph{e.g.}~\cite[Theorem~2.13]{Tian00}, to conclude that $\wt X$ is a torus quotient.
Hence, so is~$X$ by combining \cref{lem:galois closure} with the fact that an \'etale cover of a torus is again a torus.
\end{proof}

\begin{prp} \label{prop:holonomy flat factor}
In \cref{setup2}, assume that $X$ admits a \gkp cover.
Then the equality $\rk \sF = \wt q(X)$ holds.
\end{prp}

\begin{proof}
After replacing $X$ with a \gkp cover, we can assume that $\piet{\Reg X} \isom \piet X$.
We may furthermore replace $X$ by a torus cover $T \x Z \to X$.
It is then sufficient to show that $\T Z$ has no flat factor.
In other words, we may replace $X$ by $Z$ and we need to show that $\wt q(X) = 0$ implies $\rk \sF = 0$.

The flat factor $\sF$ is given by a finite dimensional representation $\rho \from \pi_1(\Reg X)\to \mathrm{SU}(r)$, with $r \defn \rk \sF$.
This representation factors thus through the fundamental group of $X$ (see Remark~\ref{rem:extension}) and it yields $\rho_X \from \pi_1(X)\to \mathrm{SU}(r)$.
Let us consider $\wh X \to X$ a desingularization of $X$: this compact \kahler manifold has vanishing Kodaira dimension and is thus special (in the sense of Campana, \emph{cf.}~\cite[Theorem~5.1]{Campana04}).
Since $X$ has canonical singularities, we know that its fundamental group is unchanged\footnote{\cite[Theorem~1.1]{Takayama2003} is only stated for projective varieties but it holds in the complex analytic setting, its proof being completely local (in the Euclidean topology). It can also be noted that the corresponding result for the \'etale fundamental group is~\cite[Theorem~7.5]{Kol93} and that this result is explicitly stated for normal analytic spaces.} when passing to a smooth model by~\cite[Theorem~1.1]{Takayama2003} hence we can interpret the representation $\rho_X$ as a representation $\rho_{\wh X} \from \pi_1(\wh X) \to \mathrm{SU}(r)$.
The manifold $\wh X$ being special, we know that none of its \'etale covers dominate a variety of general type and~\cite[Theorem~6.5]{CCE15} implies that the image of $\rho_{\wh X}$ is virtually abelian and, in particular, $\img \rho$ is virtually abelian as well.
By \cref{lem:virab}, $\rho$ has finite image.
This means that one can construct a \qe cover $f \from Y \to X$ such that the direct summand $f^{[*]}\sF$ of $\sT_Y$ satisfies $f^{[*]} \sF \isom \O Y^{\oplus r}$.
As $q(Y) = 0$, this can only occur if $r = 0$.
\end{proof}

\begin{rem} \label{rem:existence GKP cover}
The existence of \gkp covers should be true for any compact complex space with klt singularities (or even for Zariski open subsets of such spaces).
In~\cite{GKP16}, the algebraicity assumption is mainly used when appealing to~\cite{Xu14} where it is shown that $\piet{\mathrm{Link}(X, x)}$ is finite for $(X, x)$ klt.
In that article, it is proven that it is possible to extract a \emph{Koll\'ar component}, \emph{i.e.}~a birational morphism $\mu \from Y \to X$ such that the fiber $\mu\inv(x)$ has a natural structure of \Q-Fano variety.
This is achieved by considering a (projective) desingularization $\hat{\mu} \from \wh X \to X$ and running a well-chosen MMP $\wh X \dashrightarrow Y$ over $X$.
It is thus an urgent task to generalize the known results~\cite{BCHM10} to obtain a relative MMP for projective morphisms between normal complex spaces.
\end{rem}

\subsection{Varieties with strongly stable tangent sheaf} \label{section strongly stable variety}

The following definition originates in the projective setting from~\cite{GKP}. 

\begin{dfn}[ICY and IHS varieties]
Let $X$ be a compact \kahler space of dimension $n \ge 2$ with canonical singularities and $\can X \isom \O X$.

\begin{enumerate}
\item We call $X$ \emph{irreducible Calabi--Yau (ICY)} if $\HH0.Y.\Omegar Yp. = 0$ for all integers $0 < p < n$ and all \qe covers $Y \to X$, in particular for $X$ itself.
\item We call $X$ \emph{irreducible holomorphic symplectic (IHS)} if there exists a holomorphic symplectic two-form $\sigma \in \HH0.X.\Omegar X2.$ such that for all \qe covers $\gamma \from Y \to X$, the exterior algebra of global reflexive differential forms is generated by $\gamma^{[*]} \sigma$.
\end{enumerate}
\end{dfn}

\begin{dfn}[Strong stability] \label{strongly stable}
Let $X$ be a compact \kahler space of dimension $n \ge 2$ with klt singularities and trivial first Chern class.
We say that $\T X$ is \emph{strongly stable} if for any \qe cover $f \from Y \to X$, the tangent sheaf $\T Y$ is stable with respect to any \kahler class $a \in \HH2.Y.\R.$.
\end{dfn}

Using \cref{defn stability}, one can rephrase the above definition by saying that $\sT_X$ is strongly stable with respect to any \kahler class. One can show along the same lines as the proof of the third item of \cref{intro:hol} that $\sT_X$ is strongly stable if and only if it is strongly stable with respect to a single \kahler class.  \\

%
%
%

Putting together \cref{bochner}, \cref{corhol} and \cref{flat tangent implies torus}, we get the following result.

\begin{cor}[Spaces with strongly stable tangent sheaf] \label{TX ss structure}
Let $X$ be a compact \kahler space with klt singularities and trivial first Chern class of dimension $n \ge 2$.
If $\T X$ is strongly stable, then $X$ admits a \qe cover that is either a ICY or an IHS variety.
\end{cor}

\begin{proof}
By ~\cref{abundance} one can assume that $X$ has trivial canonical bundle and canonical singularities.
Given that $\T X$ is strongly stable, the decomposition~\labelcref{decompt} on a cover $Y\to X$ reduces to a single factor.
If that factor is the flat factor $\sF$, then $X$ is a torus quotient by \cref{flat tangent implies torus}, which is impossible since $\T X$ is strongly stable and $\dim X \ge 2$.
Therefore, we can apply \cref{corhol} to see that the holonomy of $\Reg{Y}$ is either $\mathrm{SU}(n)$ or $\mathrm{Sp}(n/2)$.
By standard results of representation theory of the latter groups, the statement follows from the Bochner principle, \emph{i.e.}~\cref{bochner}.
\end{proof}

\begin{rem}
For the proof of \cref{TX ss structure} above, we do not need the full Bochner principle since we only need to understand reflexive differential forms, \emph{cf.}~\cref{poles}.
The Bochner principle for more general tensors will however be applied in the proof of \cref{chi ne 0 pi1}.
\end{rem}

\section{Fundamental groups} \label{sec pi1}

This section is devoted to obtaining consequences about the fundamental groups of $X$ and of $\Reg X$, where $X$ has vanishing first Chern class.
In a first step, we provide a sufficient criterion for $\pi_1(X)$ to be finite.
This is the analogue of~\cite[Proposition~8.23]{GKP} for \kahler spaces, and as in the projective case the proof is an application of~\cite[Corollary~5.3]{Cam95}.
However, instead of using Miyaoka's Generic Semipositivity Theorem, we rely on the Bochner principle, \cref{bochner}.

\begin{prp}[Finiteness criterion for $\pi_1$] \label{chi ne 0 pi1}
Let $X$ be a compact \kahler space of dimension $n \ge 1$ with klt singularities and $\cc1X = 0$.
Assume moreover that $\chi(X, \O X) \ne 0$.
Then $\pi_1(X)$ is finite, of cardinality
\[ \big| \pi_1(X) \big| \le \frac{2^{n-1}}{|\chi(X, \O X)|}. \] 
\end{prp}

The proof is completely parallel to the one of~\cite[Theorem~3.5]{Cam20} (for $X$ projective) once we have Bochner's principle at our disposal, and we do not include it here for sake of brevity.

In terms of spaces with strongly stable tangent sheaf (\cref{TX ss structure}), this criterion already applies to all even-dimensional spaces (see below). For odd-dimensional ICYs, we have no results about $\pi_1(X)$ itself, but only about its representation theory (and the one of $\pi_1(\Reg X)$). These results are actually valid for all spaces with vanishing augmented irregularity.

\begin{thm}[Fundamental groups of Ricci flat spaces] \label{pi1 c1=0}
Let $X$ be a compact \kahler space with klt singularities and $\cc1X = 0$.
\begin{enumerate}
\item If $\dim X$ is even and $\T X$ is strongly stable, then $\pi_1(X)$ is finite. If $X$ is IHS or an even-dimensional ICY, then $X$ is even simply connected.
\item If $\wt q(X)=0$, any complex linear representation $\pi_1(X) \to \GL r\C$ has finite image, regardless of the parity of $\dim X$.
\end{enumerate}
\end{thm}

This first two items in the statement above is a straightforward consequence of \cite[Theorem~I]{GGK} and Bochner's principle. 

Restricting ourselves to dimension four, we have a completely unconditional result:

\begin{thm}[Fundamental groups in dimension four] \label{pi1 dim four}
Let $X$ be a compact \kahler space of dimension $\le 4$ with klt singularities and $\cc1X = 0$.
Then:
\begin{enumerate}
\item\label{pf.1} $\pi_1(X)$ is virtually abelian, \emph{i.e.}~it contains an abelian (normal) subgroup $\Gamma \trianglelefteq \pi_1(X)$ of finite index.
\item\label{pf.2} All finite index abelian subgroups of $\pi_1(X)$ have even rank at most $2 \, \wt q(X) \le 8$.
In particular, if $\wt q(X) = 0$ then $\pi_1(X)$ is finite.
\end{enumerate}
\end{thm}

\begin{proof}[Proof of \cref{pi1 dim four}]
Recall from~\cite[Corollary~1.8]{AlgApprox} that all statements are well-known if $\dim X \le 3$.
We can therefore assume for the remainder that $X$ is of dimension four.
Let $\gamma \from T \x Z \to X$ be a torus cover (\cref{torus cover}).
By \emph{e.g.}~\cite[Proposition~1.3]{Campana91}, the image of $\pi_1(T \x Z) = \pi_1(T) \x \pi_1(Z)$ in $\pi_1(X)$ has finite index, and $\pi_1(T)$ is free abelian of rank $2 \, \wt q(X)$.
It is therefore sufficient to show that $\pi_1(Z)$ is finite.

If $\dim Z \le 3$, this holds by~\cite[Corollary~1.8]{AlgApprox} again.
It remains to consider the case where $\dim Z = 4$, \emph{i.e.}~where $X = Z$ and $\can X \isom \O X$ and $\wt q(X) = 0$.
By \cref{rem KS}, the last property implies $\HH1.X.\O X. = 0$ and then also $\HH3.X.\O X. \isom \HH3.X.\can X. = \HH1.X.\O X.\dual = 0$ by Serre duality~\cite[Chapter~VII, Theorem~3.10]{BS76}.
So
\[ \chi(X, \O X) = 1 + \hh2.X.\O X. + \underbrace{\hh4.X.\O X.}_{= 1} \ge 2. \]
By \cref{chi ne 0 pi1}, $\pi_1(X)$ is finite (of cardinality at most $4$).
This settles~\labelcref{pf.1}.

We have already exhibited a finite index abelian subgroup of $\pi_1(X)$ of rank at most $2 \, \wt q(X)$, namely the image of $\pi_1(T) \to \pi_1(X)$.
It is well-known that all such subgroups have the same rank.
Claim~\labelcref{pf.2} follows easily.
\end{proof}

\bibliographystyle{alpha}
\bibliography{biblio,20_Literatur}

\begin{thebibliography}{BCHM10}

\bibitem[Art69]{Artin69}
M.~Artin.
\newblock Algebraic approximation of structures over complete local rings.
\newblock {\em Inst. Hautes \'Etudes Sci. Publ. Math.}, (36):23--58, 1969.

\bibitem[BCHM10]{BCHM10}
Caucher Birkar, Paolo Cascini, Christopher~D. Hacon, and James McKernan.
\newblock Existence of minimal models for varieties of log general type.
\newblock {\em J.~Amer.~Math.~Soc.}, 23:405--468, 2010.

\bibitem[Bea83]{Bea83}
Arnaud Beauville.
\newblock Vari\'{e}t\'{e}s {K}\"{a}hleriennes dont la premi\`{e}re classe de
  {C}hern est nulle.
\newblock {\em J. Differential Geom.}, 18(4):755--782 (1984), 1983.

\bibitem[BEG13]{BEG}
S\'ebastien Boucksom, Philippe Eyssidieux, and Vincent Guedj.
\newblock Introduction.
\newblock In {\em An introduction to the {K}\"ahler-{R}icci flow}, volume 2086
  of {\em Lecture Notes in Math.}, pages 1--6. Springer, Cham, 2013.

\bibitem[BGL20]{BGL}
Benjamin Bakker, Henri Guenancia, and Christian Lehn.
\newblock {Algebraic approximation and the decomposition theorem for K\"{a}hler
  Calabi-Yau varieties}.
\newblock \href{http://arxiv.org/abs/2012.00441}{arXiv:2012.00441 [math.AG]},
  2020.

\bibitem[BS76]{BS76}
Constantin B{\u a}nic{\u a} and Octavian St{\u a}n{\u a}{\c s}il{\u a}.
\newblock {\em {Algebraic Methods in the Global Theory of Complex Spaces}}.
\newblock John Wiley \& Sons, 1976.

\bibitem[Cam91]{Campana91}
F.~Campana.
\newblock On twistor spaces of the class {$\mathscr C$}.
\newblock {\em J. Differential Geom.}, 33(2):541--549, 1991.

\bibitem[Cam95]{Cam95}
Fr{\'e}d{\'e}ric Campana.
\newblock Fundamental group and positivity of cotangent bundles of compact
  {K}\"{a}hler manifolds.
\newblock {\em J. Algebraic Geom.}, 4(3):487--502, 1995.

\bibitem[Cam04]{Campana04}
Fr\'{e}d\'{e}ric Campana.
\newblock Orbifolds, special varieties and classification theory.
\newblock {\em Ann. Inst. Fourier (Grenoble)}, 54(3):499--630, 2004.

\bibitem[Cam21]{Cam20}
Fr\'{e}d\'{e}ric Campana.
\newblock The {B}ogomolov-{B}eauville-{Y}au decomposition for {KLT} projective
  varieties with trivial first {C}hern class---without tears.
\newblock {\em Bull. Soc. Math. France}, 149(1):1--13, 2021.

\bibitem[Car57]{Cartan57}
Henri Cartan.
\newblock Quotient d'un espace analytique par un groupe d'automorphismes.
\newblock In {\em Algebraic geometry and topology}, pages 90--102. Princeton
  University Press, Princeton, N.~J., 1957.
\newblock A symposium in honor of S.~Lefschetz.

\bibitem[CCE15]{CCE15}
Fr\'{e}deric Campana, Beno\^{\i}t Claudon, and Philippe Eyssidieux.
\newblock Repr\'{e}sentations lin\'{e}aires des groupes k\"{a}hl\'{e}riens:
  factorisations et conjecture de {S}hafarevich lin\'{e}aire.
\newblock {\em Compos. Math.}, 151(2):351--376, 2015.

\bibitem[CGP19]{JHM2}
Junyan Cao, Henri Guenancia, and Mihai P{\u{a}}un.
\newblock {{Variation of singular K{\"a}hler-Einstein metrics: Kodaira
  dimension zero}}.
\newblock Preprint \href{http://arxiv.org/abs/1908.08087}{arXiv:1908.08087}, to
  appear in J. Eur. Math. Soc., 2019.

\bibitem[Che55]{Chevalley55}
Claude Chevalley.
\newblock Invariants of finite groups generated by reflections.
\newblock {\em Amer. J. Math.}, 77:778--782, 1955.

\bibitem[Dem85]{Dem85}
Jean-Pierre Demailly.
\newblock Mesures de {M}onge-{A}mp{\`e}re et caract{\'e}risation
  g{\'e}om{\'e}trique des vari{\'e}t{\'e}s alg{\'e}briques affines.
\newblock {\em M{\'e}m. Soc. Math. France (N.S.)}, (19):124, 1985.

\bibitem[DG94]{DG94}
G.~Dethloff and H.~Grauert.
\newblock Seminormal complex spaces.
\newblock In {\em Several complex variables, {VII}}, volume~74 of {\em
  Encyclopaedia Math. Sci.}, pages 183--220. Springer, Berlin, 1994.

\bibitem[Dru14]{DruelZL}
St{\'e}phane Druel.
\newblock The {Z}ariski-{L}ipman conjecture for log canonical spaces.
\newblock {\em Bull. Lond. Math. Soc.}, 46(4):827--835, 2014.

\bibitem[Dru18]{Dru16}
St\'ephane Druel.
\newblock A decomposition theorem for singular spaces with trivial canonical
  class of dimension at most five.
\newblock {\em Invent. Math.}, 211(1):245--296, 2018.

\bibitem[EGZ09]{EGZ}
Philippe Eyssidieux, Vincent Guedj, and Ahmed Zeriahi.
\newblock {Singular K{\"a}hler-Einstein metrics}.
\newblock {\em {J. Amer. Math. Soc.}}, 22:607--639, 2009.

\bibitem[Fuj78]{Fuj78}
Akira Fujiki.
\newblock {On Automorphism Groups of Compact K{\"a}hler Manifolds}.
\newblock {\em Invent.~Math.}, 44:225--258, 1978.

\bibitem[GGK19]{GGK}
Daniel Greb, Henri Guenancia, and Stefan Kebekus.
\newblock Klt varieties with trivial canonical class: holonomy, differential
  forms, and fundamental groups.
\newblock {\em Geom. Topol.}, 23:2051--2124, 2019.

\bibitem[GK14]{GK13}
Patrick Graf and S{\'a}ndor~J. Kov{\'a}cs.
\newblock {An optimal extension theorem for $1$-forms and the
  {L}ipman--{Z}ariski Conjecture}.
\newblock {\em Documenta Math.}, 19:815--830, 2014.

\bibitem[GK20]{TorusQuotients}
Patrick Graf and Tim Kirschner.
\newblock {Finite quotients of three-dimensional complex tori}.
\newblock {\em Ann.~Inst.~Fourier (Grenoble)}, 70(2):881--914, 2020.

\bibitem[GKKP11]{GKKP11}
Daniel Greb, Stefan Kebekus, S{\'a}ndor~J. Kov{\'a}cs, and {\relax Th}omas
  Peternell.
\newblock Differential forms on log canonical spaces.
\newblock {\em Publications Math{\'e}matiques de L'IH{\'E}S}, 114:1--83, 2011.

\bibitem[GKP16a]{GKP16}
Daniel Greb, Stefan Kebekus, and Thomas Peternell.
\newblock \'{E}tale fundamental groups of {K}awamata log terminal spaces, flat
  sheaves, and quotients of abelian varieties.
\newblock {\em Duke Math. J.}, 165(10):1965--2004, 2016.

\bibitem[GKP16b]{GKP}
Daniel Greb, Stefan Kebekus, and Thomas Peternell.
\newblock Singular spaces with trivial canonical class.
\newblock In {\em Minimal Models and Extremal Rays, Kyoto, 2011}, volume~70 of
  {\em Adv.~Stud.~Pure Math.}, pages 67--113. Mathematical Society of Japan,
  Tokyo, 2016.

\bibitem[GM88]{GMbook}
Mark Goresky and Robert MacPherson.
\newblock {\em Stratified {M}orse theory}, volume~14 of {\em Ergebnisse der
  Mathematik und ihrer Grenzgebiete (3) [Results in Mathematics and Related
  Areas (3)]}.
\newblock Springer-Verlag, Berlin, 1988.

\bibitem[GR84]{CAS}
Hans Grauert and Reinhold Remmert.
\newblock {\em Coherent analytic sheaves}, volume 265 of {\em Grundlehren der
  Mathematischen Wissenschaften [Fundamental Principles of Mathematical
  Sciences]}.
\newblock Springer-Verlag, Berlin, 1984.

\bibitem[Gra18]{AlgApprox}
Patrick Graf.
\newblock {Algebraic approximation of K{\"a}hler threefolds of Kodaira
  dimension zero}.
\newblock {\em Math.~Annalen}, 371:487--516, 2018.

\bibitem[Gue16]{GSS}
Henri Guenancia.
\newblock {Semistability of the tangent sheaf of singular varieties}.
\newblock {\em Algebraic Geometry}, 3(5):508--542, November 2016.

\bibitem[HP19]{HP}
Andreas H\"{o}ring and Thomas Peternell.
\newblock Algebraic integrability of foliations with numerically trivial
  canonical bundle.
\newblock {\em Invent. Math.}, 216(2):395--419, 2019.

\bibitem[Kaw81]{Kaw81}
Yujiro Kawamata.
\newblock Characterization of abelian varieties.
\newblock {\em Compos.~Math.}, 43(2):253--276, 1981.

\bibitem[Kaw85]{Kaw85}
Yujiro Kawamata.
\newblock {Minimal models and the Kodaira dimension of algebraic fiber spaces}.
\newblock {\em J.~reine angew.~Math.}, 363:1--46, 1985.

\bibitem[KM98]{KM98}
J{\'a}nos Koll{\'a}r and Shigefumi Mori.
\newblock {\em Birational Geometry of Algebraic Varieties}, volume 134 of {\em
  Cambridge Tracts in Mathematics}.
\newblock Cambridge University Press, Cambridge, 1998.

\bibitem[Kol93]{Kol93}
J{\'a}nos Koll{\'a}r.
\newblock {Shafarevich maps and plurigenera of algebraic varieties}.
\newblock {\em Invent.~Math.}, 113:177--215, 1993.

\bibitem[Kol07]{Kol07}
J{\'a}nos Koll{\'a}r.
\newblock {\em {Lectures on resolution of singularities}}, volume 166 of {\em
  Annals of Mathematics Studies}.
\newblock Princeton University Press, Princeton, NJ, 2007.

\bibitem[KS21]{KS}
Stefan Kebekus and Christian Schnell.
\newblock Extending holomorphic forms from the regular locus of a complex space
  to a resolution of singularities.
\newblock {\em J. Amer. Math. Soc.}, 34(2):315--368, 2021.

\bibitem[Mat73]{Mather73}
John~N. Mather.
\newblock Stratifications and mappings.
\newblock In {\em Dynamical systems ({P}roc. {S}ympos., {U}niv. {B}ahia,
  {S}alvador, 1971)}, pages 195--232, 1973.

\bibitem[Mat12]{Mather70}
John Mather.
\newblock Notes on topological stability.
\newblock {\em Bull. Amer. Math. Soc. (N.S.)}, 49(4):475--506, 2012.

\bibitem[P{\u{a}}u08]{Paun}
Mihai P{\u{a}}un.
\newblock {Regularity properties of the degenerate Monge-Amp{\`e}re equations
  on compact K{\"a}hler manifolds.}
\newblock {\em Chin. Ann. Math., Ser. B}, 29(6):623--630, 2008.

\bibitem[ST54]{ShephardTodd54}
G.~C. Shephard and J.~A. Todd.
\newblock Finite unitary reflection groups.
\newblock {\em Canad. J. Math.}, 6:274--304, 1954.

\bibitem[Tak03]{Takayama2003}
Shigeharu Takayama.
\newblock Local simple connectedness of resolutions of log-terminal
  singularities.
\newblock {\em Internat. J. Math.}, 14(8):825--836, 2003.

\bibitem[Tia00]{Tian00}
Gang Tian.
\newblock {\em Canonical metrics in {K}\"{a}hler geometry}.
\newblock Lectures in Mathematics ETH Z\"{u}rich. Birkh\"{a}user Verlag, Basel,
  2000.
\newblock Notes taken by Meike Akveld.

\bibitem[Tou68]{Tou68}
Jean-Claude Tougeron.
\newblock Id\'{e}aux de fonctions diff\'{e}rentiables. {I}.
\newblock {\em Ann. Inst. Fourier (Grenoble)}, 18(fasc., fasc. 1):177--240,
  1968.

\bibitem[TT81]{LeTeissier}
L\^{e}~D{\~u}ng Tr\'{a}ng and Bernard Teissier.
\newblock Vari\'{e}t\'{e}s polaires locales et classes de {C}hern des
  vari\'{e}t\'{e}s singuli\`eres.
\newblock {\em Ann. of Math. (2)}, 114(3):457--491, 1981.

\bibitem[Var89]{Var}
Jean Varouchas.
\newblock {K{\"a}hler spaces and proper open morphisms.}
\newblock {\em Math. Ann.}, 283(1):13--52, 1989.

\bibitem[Voi02]{Voi02}
Claire Voisin.
\newblock {\em {Hodge Theory and Complex Algebraic Geometry I}}, volume~76 of
  {\em Cambridge Studies in Advanced Mathematics}.
\newblock Cambridge University Press, 2002.

\bibitem[Wan16]{WangB}
Botong Wang.
\newblock Torsion points on the cohomology jump loci of compact {K}\"ahler
  manifolds.
\newblock {\em Math. Res. Lett.}, 23(2):545--563, 2016.

\bibitem[Xu14]{Xu14}
Chenyang Xu.
\newblock Finiteness of algebraic fundamental groups.
\newblock {\em Compos. Math.}, 150(3):409--414, 2014.

\bibitem[Yau78]{Yau78}
Shing-Tung Yau.
\newblock On the {R}icci curvature of a compact {K}{\"a}hler manifold and the
  complex {M}onge-{A}mp{\`e}re equation. {I}.
\newblock {\em Comm. Pure Appl. Math.}, 31(3):339--411, 1978.

\end{thebibliography}

\end{document}